\documentclass[11pt,a4paper,leqno]{article}
\usepackage{graphicx}
\usepackage{makeidx}
\usepackage{pst-node}
\usepackage{amsmath,amsthm}
\usepackage{enumerate}
\usepackage{verbatim}
\usepackage{framed}
 \usepackage{ulem}

\newtheorem{theorem}{Theorem}
\newtheorem{lemma}[theorem]{Lemma} 
\newtheorem{proposition}[theorem]{Proposition} 
\newtheorem{corollary}[theorem]{Corollary}

\theoremstyle{definition}
\newtheorem{remark}{Remark}

\def\C{\bf \mbox{I\hspace{-.47em}C}}

\title{
A Dolbeault Lemma for Temperate Currents.
}

 \date{21.05.2020}

\author{Henri Skoda \thanks{Corresponding author: Henri Skoda, Sorbonne University, IMJ-PRG
Campus Pierre et Marie Curie,
4 Place Jussieu, 75005 Paris France, E-mail: henri.skoda@imj-prg.fr}
}

\begin{document}
\def\R{{\bf \mbox{I\hspace{-.17em}R}}}
\def\N{\mbox{I\hspace{-.15em}N}}

\maketitle
 
  \hspace{2cm} $Dedicated\, to\, the\, memory\, of\, Pierre\, Dolbeault. $\\

\begin{abstract}
We  consider a bounded open Stein subset $\Omega$ of a complex Stein manifold $X$ of dimension $n$. 
We prove that if $f$ is a current on $X$ of bidegree $(p,q+1)$, $\overline\partial$-closed on $\Omega$,
 we can find  a current $u$ on $X$ of bidegree $(p,q)$ which is a  solution of the equation 
$\overline\partial u=f$ in $\Omega$. 
In other words, we prove that the Dolbeault complex of temperate currents on $\Omega$ 
(\emph{i.e.} currents on $\Omega$ which extend to currents on $X$)
is concentrated in degree $0$. 
Moreover if $f$ is a current on $X=\C^n$ of order $k$, then we can find a solution $u$ which is a current on $\C^n$ of order $k+2n+1$. 
\end{abstract}

\textbf{Keywords}: Stein open subset of $\C^n$ or of a Stein manifold, $L^2$ estimates, $\bar\partial$-operator, Dolbeault $\bar\partial$-complex, temperate distributions and currents, temperate cohomology, Sobolev spaces.

\section{Introduction} \label{SectionIntroduction}

We will prove the following  result in the same way as the famous  `` Dolbeault-Grothendieck '' lemma for $\bar\partial$.

\begin{theorem}
\label{partialtemperedchomology}
 Let $\Omega$ be a bounded Stein open subset of $\C^n$ 
and let $f$ be a given current of bidegree $(p,q+1)$ on $\C^n$ 
(with compact support) 
which is $\bar\partial$-closed on $\Omega$.
Then there exists a current $u$ of bidegree $(p,q)$ 
(with compact support) in $\C^n$ such that:
\begin{equation}\label{}
\bar\partial u=f,
\end{equation}
$in\,\Omega$.\\
Moreover if $f$ is of order $k$ (resp. if $f\in H^{-s}_{(p,q+1)}(\C^n)$  for some $s> 0$), 
we can find a solution $u$ is of order at most $k+2n+1$
  (resp. $u\in H_{(p,q)}^{-s-2n-1}(\C^n)$, more precisely if $k$ is the integer such that:
 $s\le k<s+1$, for every $r>k$, we can find $u\in H_{(p,q)}^{-r-2n}(\C^n)$).
\end{theorem}
We say that a current $T$ on $\Omega\subset\C^n$ is temperate
 if and only if it can be extended to $\C^n$.
In other words, we have: 

\begin{corollary}\label{partialtemperedchomology2}
For a given relatively compact open  Stein subset of $\C^n$, the Dolbeault $\bar\partial$-cohomology of temperate currents on $\Omega$ vanishes.

\end{corollary}

As usual, we denote by $H^{s}_{(p,q)}(\C^n)$  the space of current on $\C^n$ of bidegree $(p,q)$ the coefficients of which are distributions in the Sobolev space $H^s(\C^n)$. A distribution $T\in\mathcal D'(\R^n)$ is of order $k\in\N$ if it is locally a finite linear combination of derivatives of  order at most $k$ of Radon measures on $\R^n$ or equivalenty if $T$ can be extended as a continuous linear form defined on all functions of class $C^k$ with compact support in $\R^n$ or equivalently if for every relatively compact open subset $\Omega\subset \R^n$, all functions $\phi\in {D} (\Omega)$ verify an inequality : $\vert<T,\phi>\vert\le C(\Omega,T)\sup_{x\in\,\Omega} 
\Sigma_{\vert\alpha\vert\le k}\; \vert D_{x}^{\alpha}\phi(x)\vert$, 
 in which the constant $C(\Omega,T)$ only depends on $\Omega$ and $T$. 
Of course a current is of order $k$, if its coefficients are distributions of order $k$.\\

The preceding results are still valid replacing $\C^n$  by a Stein manifold 
(section \ref{Proofoftheorem2}, theorem \ref{partialtemperedchomologyStein2}) and for currents taking their values in a given holomorphic vector bundle. 
But for the sake of simplicity, we begin with the  case of $\C^n$ as in the Dolbeault-Grothendieck lemma : 
the general case of a Stein manifold does only need  more difficult technical tools 
but no truly new ideas or methods. In the case of a Stein manifold the loss of regularity is larger than $2n+1$ because we have to iterate several times the construction made in the case of $\C^n$.

This result answers a question raised by Pierre Schapira in a personal discussion. 
He hopes it can be useful to make significant progress 
in the Microlocal Analysis theories 
highlighted for instance in the papers of M. Kashiwara and P. Schapira, [KS1996] and [Scha2017] 
in which such a temperate cohomology  naturally appears.\\
Even though the result is essentially a consequence of L. H\"ormander's $L^2$ estimates for $\bar\partial$ 
(corollary \ref{beanopenpseudoconvex3}),
 it seems that it  can not be explicitly found in the literature on the subject (with complete proof). 
Let us observe the following features of the result.
 No assumption of smoothness is required for $\Omega$.
The given current current $f$ and the solution $u$ have coefficients in spaces of distribution $H^s(\C^n)$ with $s<0$.
Hence they are never supposed to be smooth 
but with temperate singularities as for instance derivatives of Dirac measures
 and the result is quite different
  from  the most usual regularity results for $\bar\partial$ involving $C^k$ regularity up to the boundary of $\Omega$ ($k\ge0$)
	both for $\Omega$ and for the given differential forms on $\Omega$.
	If $f\in H^s_{(p,q+1)}(\C^n)$ for some $s\ge0$, then $f\in L^2_{(p,q+1)}(\Omega)$
	and the result is an immediate consequence of H\" ormander's theorem
	which provides a solution $u$ in $L^2_{(p,q)}(\Omega)$. 
	Then $u$ has a trivial extension in
$ L^2_{(p,q)}(\C^n)$ (by $0$ outside $\Omega$).\\
The gap $2n+1$ of regularity for the solution $u$ does not depend on $\Omega$.
In the basic example $\Omega=B(0,R)\setminus H$ in which $H$ is a complex analytic hypersurface of the ball $B(0,R)$ of center 0 and radius $R$,
 the result does not depend at all on the complexity of the singularities of $H$ (and on the degree of $H$ when $H$ is algebraic). 
The gap $2n+1$ is an automatic consequence of the method of proof.
To improve the gap $2n+1$ does not seem to immediately have a major interest for the purpose in   [Scha2017].\\
We need four steps to prove theorem \ref{partialtemperedchomology}. 
At first, as P. Dolbeault in [Dol1956], by solving an appropriate Laplacian equation 
$\frac{1}{2}\Delta v=\bar\partial^{\star}f$ on $\C^n$ 
($\Delta$ is the usual Laplacian on $\C^n$ defined on differential forms and currents 
and $\bar\partial^{\star}$ is the operator adjoint of $\bar\partial$ for the  usual Hermitian structure on $\C^n$) 
and replacing $f$ by $f-\bar \partial v$, we reduce the problem to the case of a current $f$
 which has harmonic coefficients on $\Omega$. 
As $f$ is temperate, 
the mean value properties of harmonic functions imply that $f$ grows at the boundary of 
$\Omega$ like a negative power of the distance $d(z,\partial\Omega)$ to the boundary of $\Omega$ (for $z\in\Omega$).
 Then  H\"ormander's $L^2$ estimates for $\bar\partial$ give a solution $u$ of $\bar\partial u=f$ 
 such that $\int_{\Omega}\vert u\vert^2\lbrack d(z,\partial\Omega)\rbrack^{2l} d\lambda(z)<+\infty$ (for some $l>0$). 
Finally using an extension theorem of L. Schwartz [Schw1950]
 for distributions, 
$u$ can be extended as a current on $\C^n$.\\
 Similar methods, were already used by P. Lelong [Le1964] for the Lelong-Poincar\'e $\partial\bar\partial$-equation and
 by H. Skoda [Sk1971] for the $\bar\partial$-equation to obtain solutions explicitly given on $\C^n$
 by integral representations and with precise polynomial estimates.
Y.T. Siu  has already studied holomorhic functions of polynomial growth on bounded open domain of $\C^n$
 using H\"ormander's $L^2$ estimates for $\bar\partial$ in [Siu1970].\\
We establish the preliminary results we need in Section \ref{SectionPreliminarydefinitionsandresultsnNr} 
 and we prove theorem \ref{partialtemperedchomology} in Section \ref{Proofoftheorem 1}.
 We extend the results to a Stein manifold 
in section \ref{Proofoftheorem2} theorem \ref{partialtemperedchomologyStein2} 
(using J-P. Demailly's theorem \ref{beanopenpseudoconvex3SteinDem12} extending H\"ormander's results to manifolds).\\
In the case of a subanalytic bounded open Stein subset $\Omega$ in a Stein manifold $X$, Pierre Schapira in [Scha2020] 
gives independently a proof of Theorem 1 (\emph{i.e} of corollary \ref{partialtemperedchomology2}) and  of theorem \ref{partialtemperedchomologyStein2}. 
His proof is basically founded on cohomological methods which are particularly well adapted to the subanalytic case. It also heavily depends on H\"ormander's $L^2$ estimates for $\bar\partial$ that he uses in the case of a bounded Stein open subset of $\C^n$ after embedding the given Stein manifold in some space $\C^n$. He also uses Lojasiewcz inequalities and another H\"ormander's inequality for subanalytic subsets.\\
A first version of this article only treating the case of $\C^n$ was submitted to ArXiv [Sk2020] on march 2020.\\
I thank Pierre Schapira very much for raising his insightful question which has strongly motivated this research. 

\section{Preliminary definitions and  results}
\label{SectionPreliminarydefinitionsandresultsnNr}

Before proving theorem \ref{partialtemperedchomology}, we need to remind several classical results. We have  sometimes given  direct  proof to establish the results in the appropriate form we wish.\\
An open subset of $\C^n$ is called Stein if it is holomorphically convex:
for all compact $K$ in $\Omega$ the holomorphic hull $\hat K_{\Omega}$ of $K$ is compact 
($x\in\hat K_{\Omega}$ if and only if $x\in \Omega$ and
 for all holomorphic function $f$ on $\Omega$, $\vert f(x)\vert\le \max_{\xi\in K}\vert f(\xi)\vert$).\\
Let us  recall the following fundamental H\" ormander's $L^2$ existence theorem for $\bar\partial$ [H\"or1966] or [H\"or1965] .
 We can also use J.P. Demailly's book [Dem2012], Chapter VIII, paragraph 6, Theorem 6.9, p. 379.
We denote by $L^2_{(p,q)}(\Omega,loc)$ the vector space of current of bidegree $(p,q)$ in $\Omega$ the coefficients of which are in $ L^2(\Omega,loc)$ for the usual Lebesgue measure $d\lambda$ on $\C^n$. 

\begin{theorem}\label{beanopenpseudoconvex}
Let $\Omega$ be an open pseudoconvex subset of $\C^n$ and $\phi$ a plurisubharmonic function defined on $\Omega$.
 For every $g\in L^2_{(p,q+1)}(\Omega,loc )$ with $\bar\partial g=0$
 such that: $\int_{\Omega}\vert g\vert^2 e^{-\phi} d\lambda<+\infty$, there exists $u\in L^2_{(p,q)}(\Omega,loc )$ such that:
\begin{equation}\label{beanopenpseudoconvex1}
 \bar\partial u=g
\end{equation}
in $\Omega$ and:
 \begin{equation}\label{beanopenpseudoconvex2}
\int_{\Omega}\vert u\vert^2 e^{-\phi}(1+\vert z\vert^2)^{-2} d\lambda\le\frac{1}{2}
\int_{\Omega}\vert g\vert^2 e^{-\phi} d\lambda.
 \end{equation}
\end{theorem}

If $\Omega $ is bounded, $u$ verifies the $L^2$ estimate:
\begin{equation}\label{beanopenpseudoconvex3dfl}
 \int_{\Omega}\vert u\vert^2 e^{-\phi}d\lambda\le C(\Omega) 
\int_{\Omega}\vert g\vert^2 e^{-\phi} d\lambda
\end{equation}
 with $ C(\Omega):=\frac{1}{2}(1+\max_{z\in\Omega}\vert z\vert^2)^{2}.$\\
The classical Oka-Norguet-Bremerman theorem ([H\"or1966] paragraph 2.6 and theorem 4.2.8) claims that the following assertions are equivalent:\\
 1) $\Omega$ is Stein,\\ 
2) $\Omega$ is pseudoconvex: \emph{i.e} there exists a plurisubharmonic function $\phi$ on $\Omega$
 which is exhaustive (for all $c\in\R$ the subset $\lbrace z\in\Omega\vert\;\phi(z)<c\rbrace$ is relatively compact in $\Omega$),\\
3) the function $-\log d(z,\partial\Omega)$ is plurisubharmonic in $\Omega$.\\
Therefore for a given $k\ge0$, we can choose $\phi(z)=-k\log d(z,\partial\Omega)$ in the inequality (\ref{beanopenpseudoconvex3dfl})
 and we will only need to use the following special case of theorem \ref{beanopenpseudoconvex} (see also [H\"or1965] theorem 2.2.1').

\begin{corollary}\label{beanopenpseudoconvex3}
Let $\Omega$ be a bounded Stein open subset of $\C^n$ and $k\ge 0$ be a given real number.
 Then for every $g\in L^2_{(p,q+1)}(\Omega,loc )$ with $\bar\partial g=0$
 such that: $\int_{\Omega}\vert g\vert^2 \lbrack d(z,\partial\Omega)\rbrack^k  d\lambda<+\infty$, there exists $u\in L^2_{(p,q)}(\Omega,loc )$ such that:
\begin{equation}\label{beanopenpseudoconvex4}
 \bar\partial u=g
\end{equation}
in $\Omega$ and:
 \begin{equation}\label{beanopenpseudoconvex5}
\int_{\Omega}\vert u\vert^2  \lbrack d(z,\partial\Omega)\rbrack^k  d\lambda\le C(\Omega)
\int_{\Omega}\vert g\vert^2 \lbrack d(z,\partial\Omega)\rbrack^k d\lambda
 \end{equation}
\end{corollary}
If we denote by $L^{2,k}_{(p,q)}(\Omega)$ the space of $u\in L^2_{(p,q)}(\Omega,loc )$ 
such that\\
 $\int_{\Omega}\vert u\vert^2  \lbrack d(z,\partial\Omega)\rbrack^{2k}  d\lambda<+\infty$, by:
\begin{equation}\label{beanopenpseudoconvex6anotheL2}
L^{2,k}_{0,(p,q)}(\Omega):=\lbrace u\in L^{2,k}_{(p,q)}(\Omega)\vert\;\bar\partial u\in L^{2,k}_{(p,q+1)}(\Omega)\rbrace
\end{equation}
and by $\mathcal O^{2,k}_{(p,0)}(\Omega):=\lbrace u\in L^{2,k}_{(p,0)}(\Omega)\vert\;\bar\partial u=0\rbrace$,
 corollary \ref{beanopenpseudoconvex3} means that the following Dolbeault-complex is exact:

\begin{equation}\label{beanopenpseudoconvex6}
\begin{aligned}
0\rightarrow\mathcal O^{2,k}_{(p,0)}(\Omega)\rightarrow L^{2,k}_{0,(p,0)}(\Omega)\xrightarrow{\bar\partial}
 L^{2,k}_{0,(p,1)}(\Omega)\xrightarrow{\bar\partial}
\ldots\xrightarrow{\bar\partial}  L^{2,k}_{0,(p,q)}(\Omega)\xrightarrow{\bar\partial}
 L^{2,k}_{0,(p,q+1)}(\Omega)\xrightarrow{\bar\partial}
\ldots\\
\xrightarrow{\bar\partial} L^{2,k}_{0,(p,n)}(\Omega)
\rightarrow 0.
\end{aligned}
\end{equation}

 We also need two results of real analysis.
\begin{lemma}\label{LetwbeasdistributiononRoforderk}
Let $w$ be a distribution on $\R^n$ of order $k$ which is harmonic 
(for the usual Laplacian on $\R^n$) on the bounded open subset $\Omega$ of $\R^n$. 
Then $w$ is of polynomial growth on $\Omega$:
$\vert w(z)\vert\le C(\Omega,w)\; \lbrack\; d(z,\R^n \setminus \Omega)\rbrack^{-k-n}$
 where the constant $C(\Omega,w)$ only depends on $\Omega$ and $w$.\\
If $w\in H^{-s}(\R^n)$ for $s\ge 0$ we have: $\vert w(z)\vert\le C(\Omega,w)\; \lbrack\; d(z,\R^n \setminus \Omega)\rbrack^{-k-\frac{n}{2}}$ 
where $k$ is the integer such that $s\le k<s+1$.

\end{lemma}

\begin{proof}
Let $\rho$ be a non negative regularizing function in $\mathcal D (\R^n)$ 
which  only depends on $\vert \zeta\vert$,
  has its support in the Euclidean ball of radius $1$
 and verifies: $\int_{\R^n} \rho(\zeta) d\lambda(\zeta)=1$
where $d\lambda$ is the Lebesgue measure on $\R^n$.\\
Let $\rho_{\epsilon}(\zeta):=\frac{1}{\epsilon^{n}}\rho(\frac{\zeta}{\epsilon})$
 be the associatef family of regularizing functions in $\mathcal{D}(\R^n)$
so that $\rho_{\epsilon}$ has its support in the ball of radius $\epsilon$ 
and verifies too $\int_{\R^n} \rho_{\epsilon}(\zeta) d\lambda(\zeta)=1$.\\
As $w$ is harmonic in $\Omega$,  
for every $z\in\Omega$,  $w(z)$ coincide with its mean-value on every Euclidean sphere of center $z$ and radius $r< d(z,\partial\Omega)$. 
Therefore using Fubini's theorem we get for every $\epsilon< d(z,\partial\Omega)$:
\begin{equation}\label{Thereforeusingfubinistheorem}
  w(z)=\int_{\R^n} w(z+\zeta) \rho_{\epsilon}(\zeta) d\lambda(\zeta)
	=\int_{\R^n} w(\zeta) \rho_{\epsilon}(z-\zeta) d\lambda(\zeta).
\end{equation} 

\emph{i.e.} $w=w\star\rho_{\epsilon}$ on $\Omega_{\epsilon}:=\lbrace z\vert\, d(z,\partial\Omega)<\epsilon\rbrace$ (in which $\star$  represents a convolution product) .\\ 
Testing  $w$ as a distribution on the test function (in the variable $\zeta$ ): $\rho_{\epsilon}(z-\zeta)$ with $\epsilon< d(z,\partial\Omega)\le 1$,
 equation (\ref{Thereforeusingfubinistheorem}) becomes:
\begin{equation}\label{w(z)=w(zeta),rhoepsilon}
 w(z)=<w(\zeta),\rho_{\epsilon}(z-\zeta)>_{\zeta}.
\end{equation} 
 As $w$ is a distribution of order $k$, 
we have for every function $\phi\in\mathcal D (\R^n)$ an inequality:
\begin{equation}\label{}
 \vert< w,\phi>\vert\le C_1(w)\;\sup_{\zeta\in\,\C^n} 
\Sigma_{\vert\alpha\vert\le k}\; \vert D_{\zeta}^{\alpha}\phi(\zeta)\vert,
\end{equation}
in which $C_1(w)>0$ is a constant only depending on $w$.\\
Taking $\phi(\zeta)=\rho_{\epsilon}(z-\zeta)$, we get:
\begin{equation}\label{}
 \vert w(z)\vert\le C_1(w)\;\sup_{\vert\zeta\vert\le \epsilon} 
\Sigma_{\vert\alpha\vert\le k}\; \vert D_{\zeta}^{\alpha}\rho_{\epsilon}(z-\zeta)\vert,
\end{equation} 
and:
\begin{equation}\label{}
 \vert w(z)\vert\le C_2(w) \epsilon^{-n-k}\; ,
\end{equation} 
for some constant $C_2(w)>0$.\\
As it is true for every $\epsilon< d(z,\partial\Omega)$, 
we take the limit as $\epsilon\rightarrow d(z,\partial\Omega)$ and we get;
\begin{equation}\label{}
\vert w(z)\vert\le C_2(w)\;\lbrack d(z,\partial\Omega)\rbrack^{-l}
\end{equation} 
  with $l=n+k$ and then:

\begin{equation}
\int_{\Omega} \vert w\vert^2\;\lbrack d(z,\partial\Omega)\rbrack^{2l} d\lambda <\infty.
\end{equation}
If we now assume that $w\in H^{-s}(\R^n)$ for a given $s>0$, equation (\ref{w(z)=w(zeta),rhoepsilon}) becomes:

\begin{equation}\label{w(z)=w(zeta),rhoepsilon1}
 \vert w(z)\vert=\vert<w(\zeta),\rho_{\epsilon}(z-\zeta)>_{\zeta}\vert\le \vert\vert w\vert\vert_{H^{-s}(\R^n)}\vert\vert\rho_{\epsilon}(z-\zeta)\vert\vert_{H^s(\R^n)}.
\end{equation} 
Let $k$ be the integer defined by $s\le k< s+1$ so that (denoting as usual by $\hat\phi$ the Fourier transform of $\phi$):

 \begin{equation}\label{w(z)=w(zeta),rhoepsilon2}
 \vert\vert \phi\vert\vert^2_{H^s(\R^n)}=
\int_{\R^n}(1+\vert\xi\vert^2)^s \vert \hat\phi(\xi)\vert^2 d\lambda (\xi)
\le\int_{\R^n}(1+\vert\xi\vert^2)^k \vert \hat\phi(\xi)\vert^2 d\lambda (\xi)=\vert\vert \phi\vert\vert^2_{H^k(\R^n)}
\end{equation} 

As $k$ is an integer, the norm $\vert\vert \phi\vert\vert_{H^k(\R^n)}$ is equivalent to the sum of the $L^2$ norms of the derivatives of $\phi$ of order less or equal to $k$, we have:
\begin{equation}\label{w(z)=w(zeta),rhoepsilon3}
 \vert\vert \phi\vert\vert^2_{H^k(\R^n)}\le C_2(k)
\int_{\R^n}\sum_{\vert\alpha\vert\le k}\vert D^{\alpha}\phi\vert^2 d\lambda 
\end{equation}
We replace $\phi$ by $\phi_{\epsilon}(\zeta):=\frac{1}{\epsilon^n}\phi(\frac{\zeta}{\epsilon})$\, (with$ \epsilon\le 1$) so that we get:

\begin{equation}\label{w(z)=w(zeta),rhoepsilon4}
 \vert\vert \phi_{\epsilon}\vert\vert^2_{H^k(\R^n)}\le C_2(k) \epsilon^{-2k-n}
\lbrack\int_{\R^n}\sum_{\vert\alpha\vert\le k}\vert D^{\alpha}\phi\vert^2 d\lambda\rbrack 
\end{equation}
Using (\ref{w(z)=w(zeta),rhoepsilon2}), (\ref{w(z)=w(zeta),rhoepsilon3}) and (\ref{w(z)=w(zeta),rhoepsilon4}) with $\phi(\zeta)=\rho(z-\zeta)$ 
(for a fixed $z\in\Omega$ with $\epsilon<d(z,\partial\Omega)\le1$) we finally obtain:

\begin{equation}\label{w(z)=w(zeta),rhoepsilon5}
 \vert w(z)\vert\le C_3(k,n)\;\vert\vert w\vert\vert_{H^{-s}(\R^n)}\;\epsilon^{-k-\frac{n}{2}}
\end{equation}
and when $\epsilon\rightarrow d(z,\partial\Omega)$:
\begin{equation}\label{w(z)=w(zeta),rhoepsilon6}
 \vert w(z)\vert\le C_3(k,n)\;\vert\vert w\vert\vert_{H^{-s}(\R^n)} \;\lbrack d(z,\partial\Omega)\rbrack^{-k-\frac{n}{2}}
\end{equation}

\begin{equation}\label{w(z)=w(zeta),rhoepsilon7}
\int_{\Omega} \vert w(z)\vert^2\;\lbrack d(z,\partial\Omega)\rbrack^{2k+n} d\lambda(z) <+\infty.
\end{equation}

\end{proof}
\begin{remark}
Instead of using mean properties of harmonic functions, one can also  use the elementary solution of $\Delta$ in $\R^n$ as in [KS1996] proposition 10.1., p. 53.
\end{remark}
 We also need the following theorem of L. Schwartz (in his book on distribution theory [Schw1950]). We can also  directly use theory of Sobolev spaces.
We say that a measure $\mu$ defined  on an open bounded subset $\Omega$ of $\R^n$, is  of polynomial growth at most $l$ in $\Omega$, if $\int_{\Omega}\; d(z,\partial\Omega)^l d\vert\mu\vert(z)<+\infty$.

\begin{theorem}\label{w(z)=w(zeta),rhoepsilon9on}
 A measure of polynomial growth $l$ defined  on an open bounded subset $\Omega$ of $\R^n$ can be extended as a distribution on $\R^n$ of order at most $l$.\\
Moreover if $w\in L^2(\Omega,loc)$ verifies the estimate:
$\int_{\Omega} \vert w(z)\vert^2\;\lbrack d(z,\partial\Omega)\rbrack^{2l} d\lambda(z) <+\infty$, with $l\in\N$, then for every $r>l$, $w$ can be extended as a distribution  in $H^{-r-\frac{n}{2}}(\R^n)$ (particularly in $H^{-l-\frac{n}{2}-1}(\R^n))$ .
\end{theorem}

\begin{remark}\label{observethattheextensiontildew1}
If $\int_{\Omega} \vert w(z)\vert^2\;\lbrack d(z,\partial\Omega)\rbrack^{2l} d\lambda(z)<+\infty$,
let us observe that the extension $\tilde w$  of $w$ depends ($a\, priori$) on the choice of $r>l$. 
We will use the results ot theorem \ref{w(z)=w(zeta),rhoepsilon9on} in the case of $\C^n=\R^{2n}$ so that $\tilde w$ can be constructed in $H^{-l-n-1}$.
\end{remark}

\begin{remark}\label{observethattheextensiontildew2}
 If $\int_{\Omega} \vert w(z)\vert^2\;\lbrack d(z,\partial\Omega)\rbrack^{2l} d\lambda(z)<+\infty$,
as  $\Omega$ is bounded,
 Schwarz inequality  implies that
 $\int_{\Omega}\, d(z,\partial\Omega)^l\,\vert w(z)\vert\; d\lambda(z)<+\infty$. 
$w$ defines on $\Omega$ a measure of polynomial growth at most $l$.
  Hence the first part of theorem \ref{w(z)=w(zeta),rhoepsilon9on} implies that $w$ can be extended to $\R^n$ as a distribution of order at most $l$.
\end{remark}

\begin{proof}

In L. Schwartz's book there is no proof and no references so that we give the following proof. We  consider the subspace $F\subset\cal D(\C^n)$ of functions the derivatives of which vanish at the order $\le l-1$ in every point $\zeta\in\partial\Omega$. 
 For a given $z\in\Omega$, we choose a point $\zeta\in\partial\Omega$ such that $\vert z-\zeta\vert=d(z,\partial\Omega)$
 and we apply Taylor's formula at the point $\zeta\in\partial\Omega$, at the order $l-1$ 
(with integral remainder, cf. [H\"or1983] paragraph 1.1, formula (1.1.7'))
 to a  function $\phi\in F$ restricted to the real interval $\lbrace tz+(1-t)\zeta\vert \, t\in\R,\, 0\le t\le 1\rbrace$
 linking in $\Omega$ the point $\zeta\in\partial\Omega$ to $z\in\Omega$, so that we obtain:

\begin{equation}\label{w(z)=w(zeta),rhoepsilon8taylor}
 \phi(z)=l\int_0^1 (1-t)^{l-1}\left\lbrack\,\Sigma_{\vert\alpha\vert= l}\;  D_{}^{\alpha} \phi(\zeta+t(z-\zeta))\,\frac{(z-\zeta)^{\alpha}}{\alpha!}\,\right\rbrack dt,
\end{equation}

and then:
\begin{equation}\label{w(z)=w(zeta),rhoepsilon8}
 \vert \phi(z)\vert\le C_4(l,n)\; \lbrack d(z,\partial\Omega)\rbrack^{l} 
 \;\max_{\xi \in\bar\Omega} \; \lbrack\Sigma_{\vert\alpha\vert= l}\; \vert D_{\xi}^{\alpha} \phi(\xi)\vert\rbrack.
\end{equation}
For all functions $\phi\in F$ and all measures $\mu$ on $\Omega$
 of polynomial growth $l$, \emph{i.e} $\int_{\Omega}\; d(z,\partial\Omega)^l d\vert\mu\vert(z)<+\infty$,
(using (\ref{w(z)=w(zeta),rhoepsilon8})) we have:

\begin{equation}\label{w(z)=w(zeta),rhoepsilon9}
\vert\int_{\Omega} \phi\; d\mu\; \vert
\le C_4(l,n)\;\lbrack\,\int_{\Omega}\; d(z,\partial\Omega)^l d\vert\mu\vert\,\rbrack
 \;\max_{\xi \in\bar\Omega} \;\lbrack \Sigma_{\vert\alpha\vert\le l}\; \vert D_{\xi}^{\alpha} \phi(\xi)\vert\rbrack.
\end{equation}
 
 For a given measure $\mu$ of polynomial growth $l$, we consider the space $\mathcal E^l(\C^n)$  of functions of class $\mathrm C^l$ on $\C^n$.
We apply Hahn Banach theorem to the linear form $\phi\rightarrow\int_{\Omega} \phi\; d\mu$ 
defined on the subspace $F\subset\mathcal D(\C^n)\subset\mathcal E^l(\C^n)$ and continuous for the seminorm 
 $\max_{\xi\in\bar\Omega} \;\lbrack \Sigma_{\vert\alpha\vert\le l}\; \vert D_{\xi}^{\alpha} \phi(\xi)\vert\rbrack$. 
This linear form can be extended in a continuous linear form $T$ on $\mathcal E^l(\C^n)$, such that:
\begin{equation}\label{w(z)=w(zeta),rhoepsilon10}
\vert<T,\phi> \vert
\le C_4(l,n)\;\lbrack\int_{\Omega}\; d(z,\partial\Omega)^l d\vert\mu\vert \rbrack
 \;\max_{\xi \in\bar\Omega} \;\lbrack \Sigma_{\vert\alpha\vert\le l}\; \vert D_{\xi}^{\alpha} \phi(\xi)\vert\rbrack.
\end{equation}
for all $\phi\in\mathcal E^l(\C^n)$,
 i.e. a distribution of order $l$ on $\C^n$ (with compact support).\\
 Let us now assume that $w\in L^2(\Omega, loc)$ verifies the estimate:
\begin{equation}\label{w(z)=w(zeta),rhoepsilon11}
I_l(w):=\int_{\Omega} \vert w(z)\vert^2\;\lbrack d(z,\partial\Omega)\rbrack^{2l} d\lambda(z) <+\infty.
\end{equation}
 for some integer $l\ge 0$.\\
 For every $\phi\in F$, Cauchy-Schwarz inequality gives:
\begin{equation}\label{w(z)=w(zeta),rhoepsilon12}
\vert<w,\phi>\vert^2=\vert\int_{\Omega} w\phi\; d\lambda\:\vert^2\le
I_l(w)\; \int_{\Omega} \vert \phi(z)\vert^2\;\lbrack d(z,\partial\Omega)\rbrack^{-2l} d\lambda(z).
\end{equation}
Using inequality(\ref{w(z)=w(zeta),rhoepsilon8}), (\ref{w(z)=w(zeta),rhoepsilon12}) becomes:
\begin{equation}\label{w(z)=w(zeta),rhoepsilon13}
\vert<w,\phi>\vert^2\le
C_5(l,n,\Omega)\;I_l(w)\;
 \lbrack\max_{\xi\in\bar\Omega}\Sigma_{\vert\alpha\vert= l}\;\vert D_{\xi}^{\alpha} \phi(\xi)\vert\rbrack^2.
\end{equation}
 with $C_5(l,n,\Omega):=\lbrack C_4(l,n)\rbrack^2\,\int_{\Omega}d\lambda$.\\
For every $r>l$, we use classical Sobolev inequality:
\begin{equation}\label{w(z)=w(zeta),rhoepsilon14max}
 \max_{\xi \in\R^n} \; \Sigma_{\vert\alpha\vert\le l}\; \vert D_{\xi}^{\alpha} \phi(\xi)\vert\le C_6(r)\;\vert\vert\phi\vert\vert_{H^{r+\frac{n}{2}}}
\end{equation}
 and inequality (\ref{w(z)=w(zeta),rhoepsilon13}), so that
 we obtain:
\begin{equation}\label{w(z)=w(zeta),rhoepsilon14}
\vert <w,\phi>\vert\le
C_7(l,n,r,\Omega)\;\lbrack I_l(w)\rbrack^{\frac{1}{2}}\;\vert\vert\phi\vert\vert_{H^{r+\frac{n}{2}}}.
\end{equation}\label{w(z)=w(zeta),rhoepsilon15}
with $C_7(l,n,r,\Omega):=C_6(r)\, \lbrack C_5(l,n,\Omega)\rbrack^{\frac{1}{2}}$.\\
Using still Hahn-Banach Theorem for the linear form $\phi\rightarrow<w,\phi>$ defined on the  subspace $F$ of $H^{r+\frac{n}{2}}(\R^n)$ and continuous for the norm of $ H^{r+\frac{n}{2}}$,
we extend $w$ as a distribution $T\in H^{-r-\frac{n}{2}}$ such that: 
$\vert\vert T\vert\vert_{H^{-r-\frac{n}{2}}}\le
C_7(l,n,r,\Omega)\;\lbrack I_l(w)\rbrack^{\frac{1}{2}}$ 
(of course we can also do this extension by using orthogonal projection on the closed subspace $\bar F$ in the Hilbert space $H^{r+\frac{n}{2}}(\R^n)$).
\end{proof}

We can now prove theorem 1.

\section{Proof of theorem 1}
\label{Proofoftheorem 1} 

We follow P. Dolbeault's proof of the Dolbeault-Grothendieck lemma. A. Grothendieck's proof was different, (in some sense) more elementar
 than P. Dolbeault's proof but not useful for our present purpose. 
Of course we can suppose (w.l.o.g.) that $f$ has compact support in $\C^n$ (using a cutoff function in $\mathcal D (\R^n)$ equal to 1 in a neighborhood of $\bar \Omega$).
 Let us remind that $\C^n$ being equipped with its usual flat Hermitian metric, the Laplacian acting on differential forms and currents is  defined on $\C^n$ by:
\begin{equation}\label{beanopenpseudoconvex7delta1}
\frac{1}{2} \Delta=\frac{1}{2} (dd^{\star}+d^{\star}d)=\bar\partial\bar\partial^{\star}+\bar\partial^{\star}\bar\partial
=\partial\partial^{\star}+\partial^{\star}\partial,
\end{equation}
 so that $\frac{1}{2}\Delta f$ is the usual Laplacian on $\C^n$ acting on each coefficient of the current $f$. 
  $\bar\partial^{\star}$  (resp. $\partial^{\star}$) (resp. $d^{\star}$) 
	is the adjoint of $\bar\partial$ (resp. $\partial$) (resp. $d:=\partial+\bar\partial$) for the same constant metric on $\C^n$
 (there is no weight function).\\
At first we solve in $\C^n$ the Laplacian equation:
\begin{equation}\label{beanopenpseudoconvex7}
\frac{1}{2}\Delta v:=(\bar\partial\bar\partial^{\star}+\bar\partial^{\star}\bar\partial)\; v=\bar\partial^{\star}f.
\end{equation} 
 ($v$ and $\bar\partial^{\star}f$ are of bidegree $(p,q)$.)\\
 If we write: $f=\sum_{\vert I\vert=p,\vert J\vert=q+1}^{'} f_{I,J} dz_I\wedge d\bar z_J$ 
 ($\Sigma ^{'}$ means that we only sum on strictly increasing multi-indices $I$ and $J$), we have 
(cf. [H\"or1966] paragraph 4.1, p. 82 or 85 or [Dem2012] Chapter 6, paragraph 6.1) :
\begin{equation}\label{beanopenpseudoconvex8adjoint}
\bar\partial^{\star} f=(-1)^{p-1}\Sigma_{\vert I\vert=p,\vert K\vert=q}^{'} \;
\left\lbrack\Sigma_{j=1}^{j=n} \frac{\partial}{\partial z_j}(f_{I,jK})\right\rbrack \; dz_I\wedge d\bar z_{K}. 
\end{equation}
If $f$ is of bidegree $(0,1)$, we simply have : $f=\Sigma_{j=1}^{j=n}  f_j d \bar z_j$, $\bar\partial^{\star} f=-
\Sigma_{j=1}^{j=n} \frac{\partial f_j}{\partial z_j}$ and (\ref{beanopenpseudoconvex7}) is the Laplace equation in $\C^n$ : 
$\Sigma_{j=1}^{j=n} \frac{\partial^2}{\partial z_j \partial\bar z_j} v=\Sigma_{j=1}^{j=n} \frac{\partial f_j}{\partial z_j}$.\\
$v$ is obtained by convolution of each coefficient $\frac{\partial f_{I,jK}}{\partial z_j}$
of $\bar\partial^{\star} f$ in (\ref{beanopenpseudoconvex8adjoint}) with the elementar solution $E$ of the usual Laplacian in $\C^n$.\\
We set:
\begin{equation}\label{beanopenpseudoconvex8}
g:=f-\bar\partial v
\end{equation}
As $\bar\partial^2=0$, we have (by usual computation):
\begin{equation}\label{beanopenpseudoconvex94p} 
\frac{1}{2}\Delta(\bar\partial v)=(\bar\partial\bar\partial^{\star}+\bar\partial^{\star}\bar\partial) \bar\partial v=\bar\partial\bar\partial^{\star}\bar\partial v=\bar\partial(\bar\partial\bar\partial^{\star}+\bar\partial^{\star}\bar\partial) v
=\bar\partial(\frac{1}{2}\Delta v)
\end{equation}
 \emph{i.e.} $\bar\partial$ commute with $\Delta$ on $\C^n$.
 Using (\ref{beanopenpseudoconvex7}), we have: $\frac{1}{2}\Delta(\bar\partial v)=\bar\partial\bar\partial^{\star}f$, 
  and:
\begin{equation}\label{beanopenpseudoconvex9}
\frac{1}{2}\Delta g=\frac{1}{2}\Delta f-\frac{1}{2}\Delta (\bar\partial v)=(\bar\partial\bar\partial^{\star}+\bar\partial^{\star}\bar\partial)
\; f-\bar\partial\bar\partial^{\star}f
=\bar\partial^{\star}\bar\partial\; f.
\end{equation} 
Hence (as $\bar\partial f=0$ on $\Omega$), $g$ is harmonic on $\Omega$:
\begin{equation}\label{beanopenpseudoconvex11}
\Delta g=0.
\end{equation} $f$ being of order $k$ with compact support,
 $\bar\partial^{\star}f$ is of order $k+1$ with the same support 
(the coefficients of $\bar\partial^{\star}f$ are linear combination
 of derivatives $\frac{\partial}{\partial z_j}$ of the coefficients of $f$). 
Therefore the solution $v$ of the Laplacian equation (\ref{beanopenpseudoconvex7}) is of order at most $k$.
Indeed it is obtained by convolution:  
$E\star\frac{\partial f_{I,jK}}{\partial z_j}=\frac{\partial E}{\partial z_j}\star f_{I,jK}$
 of each coefficient $\frac{\partial f_{I,jK}}{\partial z_j}$
 of 
$\bar\partial^{\star} f$ in(\ref{beanopenpseudoconvex8adjoint})
 with the elementary solution $E:=-C_n \vert z\vert^{-2n+2}$ of $\Delta$,
the first derivatives $\frac{\partial E}{\partial z_j}$ of $E$ are $O(\vert z\vert^{-2n+1})$, then in $L^1 (\C^n,loc)$
and the convolution $\frac{\partial E}{\partial z_j}\star f_{I,jK}$
 of a function in $\L^1(\R^{2n}, loc)$ with a distribution of order $k$ and compact support 
is still of order at most $k$.
Hence $\bar\partial v$ is of order at most $k+1$ and $g=f-\bar\partial v$ is too of order at most $k+1$.\\
 We write: $g=\sum_{\vert I\vert=p,\vert J\vert=q+1} g_{I,J} dz_I\wedge d\bar z_J$  
with strictly increasing multi-indices $I$ and $J$. Let $g_{I,J}$ be a coefficient of $g$.
As $g_{I,J}$ is harmonic in $\Omega$, 
we can apply lemma \ref{LetwbeasdistributiononRoforderk} in $\R^{2n}$ to $g_{I,J}$  which is a distribution of order at  most $k+1$, we get an inequality:

\begin{equation}\label{beanopenpseudoconvex12}
 \vert g_{I,J}(z)\vert\le C_1(\Omega,g_{I,J})\; \lbrack d(z,\partial\Omega)\rbrack^{-2n-k-1}.
\end{equation} 
Hence :
\begin{equation}\label{beanopenpseudoconvex13}
\vert g(z)\vert\le C_2(\Omega,g)\;\lbrack d(z,\partial\Omega)\rbrack^{-l}
\end{equation} 
 for some constant  $C_2(\Omega,g)>0$ and $l:=2n+k+1$ (and $z\in\Omega$) and then:

\begin{equation}\label{beanopenpseudoconvex14}
\int_{\Omega} \vert g\vert^2\;\lbrack d(z,\partial\Omega)\rbrack^{2l} d\lambda <\infty
\end{equation}
 where $d\lambda$ is the Lebesgue measure on $\C^n$.\\

L. H\"ormander's $L^2$ estimates for $\bar\partial$ (corollary \ref{beanopenpseudoconvex3}) provide a solution $u$ in $\Omega$ of the equation:

\begin{equation}\label{beanopenpseudoconvex15}
\bar\partial u=g
\end{equation}
with the $L^2$ estimate:
\begin{equation}\label{beanopenpseudoconvex16}
\int_{\Omega} \vert u\vert^2\;\lbrack d(z,\partial\Omega)\rbrack^{2l} d\lambda <\infty
\end{equation}

 As $\Omega$ is bounded, Cauchy-Schwarz inequality gives the following $L^1$ estimate:
\begin{equation}\label{beanopenpseudoconvex17}
\int_{\Omega} \vert u\vert\;\lbrack d(z,\partial\Omega)\rbrack^{l} d\lambda <\infty
\end{equation}

Therefore a coefficient $u_{I,J}$ of $u$ defines a measure of polynomial growth $l$ on $\Omega$.
 Using L. Schwartz's theorem \ref{w(z)=w(zeta),rhoepsilon9on}, such a measure  (of polynomial growth $l$) defined  on $\Omega$  can be extended as a distribution on $\C^n$ (of order at most $l$) so that $u$ can be extended as a current on $\C^n$ of order at most $l$. Then $u+v$ is a current on $\C^n$ verifying:
\begin{equation}\label{beanopenpseudoconvex18}
\bar\partial( u+v)=f
\end{equation}
on $\Omega$. Moreover $u+v$ has order at most $l=k+2n+1$.\\
We now consider the case of a given $f\in H^{-s}_{(p,q+1)}(\C^n)$ for some $s\ge0$. Then $\bar\partial^{\star}f\in H^{-s-1}_{(p,q)}(\C^n)$.
 Classicaly we can find a solution $v$ of the Laplace equation (\ref{beanopenpseudoconvex7})
 in $H^{-s+1}_{(p,q)}(\C^n)$ so that $g=f-\bar\partial v$ is also in $ H^{-s}_{(p,q+1)}(\C^n)$.
 We apply lemma \ref{LetwbeasdistributiononRoforderk} to every coefficient $g_{I,J}$ of $g$  in $\C^n=\R^{2n}$ 
so that $\vert g(z)\vert\le C\; \lbrack d(z,\C^n \setminus \Omega)\rbrack^{-k-n}$ where $k$ is the integer such that $s\le k<s+1$ and therefore:
\begin{equation}
\int_{\Omega} \vert g\vert^2\;\lbrack d(z,\partial\Omega)\rbrack^{2k+2n} d\lambda <+\infty.
\end{equation}
Corollary \ref{beanopenpseudoconvex3} implies we can solve $\bar\partial u=g=f-\bar\partial v$ with the estimate:
\begin{equation}
\int_{\Omega} \vert u\vert^2\;\lbrack d(z,\partial\Omega)\rbrack^{2k+2n} d\lambda <+\infty.
\end{equation}

We now apply theorem \ref{w(z)=w(zeta),rhoepsilon9on} in $\C^n=\R^{2n}$ to every coefficient of $u$ with $l=k+n$
($l+n=k+2n$) so that for every $r>k$, u can be extended as a current in $\C^n$ in $H^{-r-2n}_{(p,q)}(\C^n)$.
As $v\in H^{-s+1}_{(p,q)}(\C^n)$, $u+v$ is too in $H^{-r-2n}_{(p,q)}(\C^n)$ and verifies $\bar\partial(u+v)=f$ in $\Omega$.

\begin{remark}\label{Proofoftheorem2remarkprecise}
We have a little more precise result:  $f=\bar\partial(u+v)$ in $\Omega$ 
with $u\in L^{2,k+n}_{(p,q)}(\Omega)\cap H^{-r-2n}_{(p,q)}(\C^n)$ 
(for every $r>k$ particularly for $r=k+1$)
 and $v\in H^{-s+1}_{(p,q)}(\C^n)$.
We will use the fact that $v\in H^{-s+1}_{(p,q)}(\C^n)$ has a better regularity than $f\in H^{-s}_{(p,q)}(\C^n)$ in section \ref{Proofoftheorem2}.

\end{remark}

\section{Extension of theorem 1 to Stein manifolds}
\label{Proofoftheorem2} 
We will now see that theorem \ref{partialtemperedchomology} 
remains true 
  for a relatively compact open Stein subset $\Omega$ of a given Stein manifold $X$.
	We can essentially use the same reasoning as in $\C^n$. But we need much stronger technical results.\\
Let us recall that a complex manifold $X$ is Stein 
if, by definition, global holomophic functions $\cal O(X)$ separate the points of $X$, 
give local holomorphic coordinates on $X$ 
and if $X$ is holomorphically convex 
(for all compact $K$ in $X$ the holomorphic hull $\hat K$ of $K$ is compact with $\hat K:=\lbrace x\in X\vert\, \vert f(x)\vert\le \max_{\xi\in K}\vert f(\xi)\vert\rbrace$)). 
Let us also remind the two following other characterizations of a Stein manifold $X$ of complex dimension $n$. 
The first one, a complex holomorphic manifold $X$ is  Stein if and only if it can be imbedded as a  closed complex submanifold of $\C^{2n+1}$.
 The second one, $X$ is Stein if and only if there exists a strictly plurisubharmonic exhaustive function $\psi$ on $X$ of class $C^2$
 (if $X$ is a closed submanifold of $\C^{2n+1}$,
 we can take for $\psi$ the restriction to $X$ 
of the function $\vert\vert x\vert\vert^2$ defined on $\C^{2n+1}$).\\
Hence $X$ is a K\"ahlerian manifold [Weil1958]
(taking, for instance, the K\"ahler metric associated with the closed K\"ahler form $i\partial\bar\partial \psi$).\\
 It is proved in [Ele75]
 that if we consider  a relatively compact Stein open subset $\Omega$ of $X$ 
and the geodesic distance associated 
with a given K\"ahlerian metric on $X$, 
then the function: $-\log d(z,\partial\Omega)+C(\Omega,\psi)\, \psi$, 
is strictly plurisubharmonic in $\Omega$
 for a constant $C(\Omega,\psi)$ large enough.\\
 Therefore using [Dem2012] Chapter VIII, paragraph 6, Theorem 6.1 p. 376 and 6.5, p. 378 or [Dem1982] 
the following result (similar to corollary \ref{beanopenpseudoconvex3}) 
still holds on a Stein manifold:\\
\begin{theorem}\label{beanopenpseudoconvex3SteinDem12}
Let $\Omega$ be a relatively compact open Stein subset of the Stein manifold $X$.
We consider on $X$ a given K\"ahler form $\omega$,
 the geodesic distance on $X$ associated with $\omega$
 and  for $z\in\Omega$ the corresponding distance $d(z,\partial\Omega)$ to the boundary of $\Omega$.
 Let us consider a holomorphic Hermitian vector bundle $F$ of rank $r$ on $X$ and currents with values in $F$.
 Let $k\ge0$ be a given real number.
Then for every $g\in L^2_{(p,q+1)}(\Omega, F,loc )$ with $\bar\partial g=0$
 such that: $\int_{\Omega}\vert g\vert^2 \lbrack d(z,\partial\Omega)\rbrack^k  d\lambda<+\infty$, there exists $u\in L^2_{(p,q)}(\Omega, F, loc )$ such that:
\begin{equation}\label{beanopenpseudoconvex4}
 \bar\partial u=g
\end{equation}
in $\Omega$ and:
 \begin{equation}\label{beanopenpseudoconvex5Stein}
\int_{\Omega}\vert u\vert^2  \lbrack d(z,\partial\Omega)\rbrack^k  d\lambda\le C(\Omega, F, k)
\int_{\Omega}\vert g\vert^2 \lbrack d(z,\partial\Omega)\rbrack^k d\lambda,
 \end{equation}
where $d\lambda=\frac{\omega^n}{n!}$ is the positive measure on $X$ 
defined by the $(n,n)$ form $\frac{\omega^n}{n!}$ 
($C(\Omega, F, k)$ is a constant $>0$ only depending on $\Omega$, $F$ and $k$).
\end{theorem}

Let us give more details about how to deduce theorem \ref{beanopenpseudoconvex3SteinDem12} 
from theorems 6.1 and 6.5 in [Dem2012].
At first let us remind Demailly's theorem 6.5 
(for the sake of simplicity we state it with a little more restrictive assumption):
\begin{theorem}\label{beanopenpseudoconvex3Stein}
Let $(X,\omega)$ be a Stein manifold $X$ of complex dimension $n$ with a given K\"ahler metric $\omega$. 
Let us consider a holomorphic  Hermitian vector bundle $F$ of rank $r$ on $X$ and a $\C^{\infty}$
 function $\phi$ on $X$ such that $ic(F)+i\partial\bar\partial\phi\ge\mu\,\omega$ 
where $c(F)$ is the curvature form of $F$ and  $\mu>0$ a given constant.  
Then for every $g\in L^2_{(n,q+1)}(X, F,loc )$ with $\bar\partial g=0$
 such that: $\int_{X}\vert g\vert^2 e^{-\phi} dV<+\infty$, there exists $u\in L^2_{(n,q)}(X, F,loc)$ such that:
\begin{equation}\label{beanopenpseudoconvex4}
 \bar\partial u=g
\end{equation}
in $X$ and:
 \begin{equation}\label{beanopenpseudoconvex5Steinhm45}
\int_{X}\vert u\vert^2 e^{-\phi}   dV\le 
\frac{1}{\mu}\int_{X}\vert g\vert^2 e^{-\phi} dV.
 \end{equation}
where $dV=\frac{\omega^n}{n!}$ is the positive measure on $X$ defined by the $(n,n)$ form $\frac{\omega^n}{n!}$.
\end{theorem}
In theorem 6.5 in [Dem2012], $F$ is a line bundle  
but the result is still valid for a vector bundle : 
you only need to consider the positivity of the curvature form $ic(F)$ of $F$ 
in the strong sense of Nakano as explained in [Dem2012] (theorem 6.1).
 For $(p,q)$-form (with $p\neq n$ and values in $F$) 
we consider (n,q)-forms with values in the new vector bundle $F\bigotimes \wedge^p T^{\star}(X)\bigotimes\wedge^n T(X)$.\\
 If now $\Omega$ is a relatively compact open Stein subset of $X$, 
we can choose a constant  $C_1(\Omega, F)$
 such that $ic(F)+C_1(\Omega, F,)\, i\partial\bar\partial\psi\ge\omega$ 
on $\bar\Omega$ (in the strong sense of Nakano). 
For every $C^{\infty}$ plurisubharmonic function $\phi$ on $\Omega$
 we can apply theorem  \ref{beanopenpseudoconvex3Stein} 
restricted to the Stein manifold $\Omega$ and the function  $C_1(\Omega, F)\,\psi+\phi$ 
so that (as $\psi$ is bounded on $\Omega$) we get an estimate:
\begin{equation}\label{beanopenpseudoconvex5Stein6YUJN}
\int_{\Omega}\vert u\vert^2 e^{-\phi}  dV\le 
C_2(\Omega,F)\int_{\Omega}\vert g\vert^2 e^{-\phi} dV.
 \end{equation} 
We can now take $\phi=-k\log d(z,\partial\Omega)+k\, C(\Omega,\psi)\, \psi$, (for some $k\ge 0$) 
so that (as $\psi$ is bounded on $\Omega$) 
we get the estimate (\ref{beanopenpseudoconvex5Stein}) of theorem \ref{beanopenpseudoconvex3SteinDem12} 
(for (p,q) forms with values in $F$). 
The function $\phi:=-k\log d(z,\partial\Omega)+k\,C(\Omega,\psi)\, \psi$ is only continuous on $\Omega$ 
but as $\phi$ is stritly plurisubharmonic on the Stein manifold $\Omega$ 
it can be closely approximated by a family $(\phi_{\epsilon })$ ($0<\epsilon<\epsilon_0$) of $C^{\infty}$ strictly plurisubharmonic functions on $\Omega$ 
as explained in [Dem2012] chapter 1, paragraph 5.E. page 42 
(Richberg theorem (5.21)) such that $\phi\le \phi_{\epsilon }\le \phi +\epsilon$. 
At first we obtain the estimate (\ref{beanopenpseudoconvex5Stein6YUJN})
 for the functions $\phi_{\epsilon}$ and  a solution $u_{\epsilon}$ of (\ref{beanopenpseudoconvex4}).  
Taking the limit as $\epsilon$ goes to 0 
and using the weak compacity of the closed ball of $L^{2,k}_{(p,q)}(X, F)$,
 we get (\ref{beanopenpseudoconvex5Stein6YUJN}) for $\phi$ and a weak limit $u$ of a subsequence of the family $(u_{\epsilon}$).\\

We will only use lemma \ref{LetwbeasdistributiononRoforderk} in a local chart of $X$ (i.e. in $\C^n$) : 
we do'nt need to extend this lemma  to the Riemannian Laplacian operator on $X$ with variable coefficients.
Replacing $\C^n=\R^{2n}$ by a complex Riemannian manifold $X$, 
extension theorem \ref{w(z)=w(zeta),rhoepsilon9on} is still valid as it is a local result 
(alongside the boundary of $\Omega$) using a partition of unity
 of class $C^{\infty}$ on $X$.\\
 We will give an appropriate simple extension of theorem \ref{partialtemperedchomology} in $\C^n$.
(proposition \ref{partialtemperedchomologyescrtyX}) 
which will be enough to  be able to work on Stein manifold $X$ 
and with an arbitrary holomorphic vector bundle on $X$.
Applying and iterating this last result in local charts of $X$ we will reduce the problem to J-P. Demailly's estimates for $\bar\partial$ (theorem \ref{beanopenpseudoconvex3SteinDem12}).\\
To make short, we set : $d_{\Omega}(z):=d(z,\partial \Omega)$.
Let us remind that for a given $k\in\R$, 
we denote by $L^{2,k}(\Omega)=L^{2,k}_{(0,0)}(\Omega)$ the space of functions $u\in L^2(\Omega,loc)$ 
such that
 $\int_{\Omega}\vert u\vert^2  \lbrack d_{\Omega}(z)\rbrack^{2k}  d\lambda<+\infty$ 
and we set: $\vert\vert u\vert\vert_k^2:=\int_{\Omega}\vert u\vert^2  \lbrack d_{\Omega}(z)\rbrack^{2k} d\lambda$.
We need the following preliminary lemma.
 
\begin{lemma}\label{w(z)=w(zeta),rhoepsilon9on23avril18WX}
  Let $\Omega$ be a bounded open subset of $\R^n$ 
	and $k\in\N$ be given. 
	Then for every $w\in L^{2,k}(\Omega)$
	there exists a solution 
	$v\in L^{2,k+n-1}(\Omega)$
	of the  equation: 
	$\Delta v=\frac{\partial}{\partial x_1} w$ 
	or of the equation : $\Delta v= w$  
	 such that :
	$\vert\vert v\vert\vert_{k+n-1}^2\le C(\Omega,k,n) \vert\vert w\vert\vert_k^2$.

\end{lemma}

We will apply lemma \ref{w(z)=w(zeta),rhoepsilon9on23avril18WX} in $\C^n=\R^{2n}$, so that $v\in L^{2,k+2n-1}(\Omega)$ 
and we will consider the equations : $\Delta v=\frac{\partial}{\partial x_j} w$, $1\le j\le n$ 
(we only refer to the equation $\Delta v= w$ for completeness).

\begin{proof} 
We can suppose that $Diam\,\Omega<1$. 
According to theorem \ref{w(z)=w(zeta),rhoepsilon9on} 
remark \ref{observethattheextensiontildew2},
$w$ can be extended to $\R^n$ as a distribution of order at most $k$ with compact support 
that we denote by $\tilde w$. 
We suppose that the support of $\tilde w$  
is a subset of an open  ball $B$ of $\R^n$ and we set $L:=\bar B$.\\
 The solution $v$ is given by the convolution 
with the elementary solution $E=-c_n\vert x\vert^{-n+2}$ of the Laplacian in $\R^n$, 
$v:=E\star (\frac{\partial}{\partial x_1}\tilde w)=(\frac{\partial}{\partial x_1}E)\star \tilde w$ 
(resp. $v:=E\star \tilde w$). Let us set $K:=\frac{\partial}{\partial x_1}E=(n-2)c_n\vert x\vert^{-n+1}$ (resp. $K=E$).
We will need the fact that $K\in L^1_{loc}(\R^n)$ 
and is of class $\mathrm C^{\infty}$ outside $0$.
We only have to prove that $v\vert_{\Omega}$ verifies 
the right estimate : $v\in L^{2,k+n-1}(\Omega)$.
 Let $\psi$ be a cutoff function in $\mathcal D(\R^n)$
 such that $0\le \psi\le 1$, $\psi=1$
 in a neighborhood of the closed ball $\bar B(0,1)$, 
 with support in the  the open ball $B(0,2)$ and such that $\psi(x)=\psi(-x)$ for all $x\in\R^n$. 
We can split $v$ 
into $v:=K\star \tilde w=(\psi K)\star \tilde w+\lbrack(1-\psi)K\rbrack\star \tilde w$.
 $(1-\psi)K\in\mathrm C^{\infty}(\R^n)$ and $\tilde w$ is a distribution 
with compact support in $\R^n$,
 then $\lbrack(1-\psi)K\rbrack\star \tilde w$ is in $\mathrm C^{\infty}(\R^n)$ 
and is bounded on the bounded subset $\bar\Omega$.
Therefore we only have to verify that $v':=(\psi K)\star \tilde w$ 
is in $L^{2,k+n-1}(\Omega)$.
 Henceforth we replace $K$ by $\tilde K:=\psi K$. 
$\tilde K$ has its support in the ball $B(0,2)$, $\tilde K(x)=\tilde K(-x)$ for all $x\in\R^n$ 
 and we  only have to consider $v':=\tilde K\star \tilde w$ instead of $v$.\\
 The idea of the proof is to locally use in $\Omega$ 
the classical $L^2$ inequality for convolution
$\vert\vert \tilde K\star w\vert\vert_{L^2}\le\vert\vert \tilde K\vert\vert_{L^1} \vert\vert w\vert\vert_{L^2}$ in $\R^n$ 
(of course after truncating $w$ by an appropriate cutoff function) (cf. [H\"or1983] Corollary 4.5.2. p. 117).
Close to the boundary $\partial\Omega$ we get another estimate, 
using the extension $\tilde w$  of $w$ as a distribution in $\R^n$.
It will give us a polynomial gap $\lbrack d_{\Omega}(z)\rbrack^{-n+1}$. 
We have to prove the following inequality 
(using duality between $L^{2,k+n-1}(\Omega)$ and $L^{2,-k-n+1}(\Omega)$):
\begin{equation}
\label{artialtemperedchomolog56laplacientemperate1UX}
\vert < \tilde K\star\tilde w,\phi>\vert^2\le C(\Omega,\tilde K,n) \vert\vert w\vert\vert_{k}^2 \vert\vert \phi\vert\vert_{-k-n+1}^2 
\end{equation}
for all $\phi \in\mathcal D(\Omega)$ or equivalently:
\begin{equation}
\label{artialtemperedchomolog56laplacientemperate2UVX}
\vert < \tilde K\star\phi,\tilde w>\vert^2\le C(\Omega,\tilde K,n) \vert\vert w\vert\vert_{k}^2 \vert\vert \phi\vert\vert_{-k-n+1}^2 
\end{equation}
for all $\phi\in\mathcal D(\Omega)$ 
(we have $<T\star S,\phi>=<\check T\star\phi,S>$
 for $S$, $T$ in $\mathcal E'(\R^n)$, 
 $\check T(x):=T(-x)$, we take $T=\tilde K$ which is symmetric 
and $S=\tilde w$).\\
For a given $\epsilon>0$, let us define $\Omega_{\epsilon}:=\,\lbrace x\in\Omega\vert\, d_{\Omega}(x)>\epsilon\rbrace$,
$K_{\epsilon}:=\lbrace x\in \Omega\vert\, d_{\Omega}(x)\ge \epsilon\rbrace$, 
$V_{\epsilon}:=\,\lbrace x\in \Omega\vert\, \epsilon< d_{\Omega}(x)<2\epsilon\rbrace$
 so that for all $x$ and $y$ in  $V_{\epsilon}$
  we have the inequality: 
\begin{equation}
\label{artialtemperedchomolog56laplacientemperate3dxdyUV}
\frac{1}{2}d_{\Omega}(y)\le d_{\Omega}(x)\le 2d_{\Omega}(y).
\end{equation}
Of course we will later take $\epsilon=2^{-j}$, $j\in\N$ 
and fulfill $\Omega$ with the exhaustive family of sets $V_{\epsilon}$.
 For a given subset $A$ of $\R^n$
 we denote by $\chi_A$ be  the characteristic  function of $A$: $\chi_A(x)=1$ if $x\in A$, 
$\chi_A(x)=0$ if $x\notin A$. 
We define $\psi_{\epsilon}=\rho_{\frac{\epsilon}{4}}\star\chi_{K_{\frac{\epsilon}{2}}}$.
$\psi_{\epsilon}\in\mathcal D(\Omega)$
is a function 
such that  $0\le\psi_{\epsilon}\le1$, $\psi_{\epsilon}=1$ 
on a neighborhood of $K_{\epsilon}:=\lbrace x\in \Omega\vert\, d_{\Omega}(x)\ge \epsilon\rbrace$, 
$Supp\,\psi_{\epsilon}\subset \Omega_{\frac{\epsilon}{4}}$ 
and $\vert D^{\alpha}\psi_{\epsilon}\vert\le C(\alpha) \epsilon^{-\vert\alpha\vert }$ 
for all multi-indices $\alpha\in N^n$ 
where $C(\alpha)$ is a constant only dependent on $\alpha$. 
Indeed we have :
$\psi_{\epsilon}(y)=\int_{K_{\frac{\epsilon}{2}}}\rho_{\frac{\epsilon}{4}}(y-x) d\lambda (x)$ 
 with $\rho_{\epsilon}(x)= \frac{1}{\epsilon^n}\rho(\frac{x}{\epsilon})$ so that
 $D^{\alpha}\psi_{\epsilon}(y)=\frac{1}{\epsilon^{n+\vert\alpha\vert}}\int_{K_{\frac{\epsilon}{2}}} (D^{\alpha}\rho)({\frac{y-x}{\epsilon}}) d\lambda (x)$
 and 
$\vert D^{\alpha}\psi_{\epsilon}(y)\vert\le\frac{1}{\epsilon^{\vert\alpha\vert}}\int_{R^n} \vert D^{\alpha}\rho\vert(x) d\lambda (x)$.\\
We also set : $\psi'_{\epsilon}=1-\psi_{\epsilon}$ 
so that $\psi_{\epsilon}+\psi'_{\epsilon}=1$ on $\Omega$.
$Supp\, \psi'_{\epsilon}\subset \lbrace x\in\Omega\vert\,d_{\Omega}(x)<\epsilon\rbrace$ 
so that  $Supp\,\psi'_{\epsilon}\cap K_{\epsilon}=\emptyset$.\\
 $\epsilon$ being fixed, 
let us assume at first that $Supp\,\phi\subset V_{\epsilon}$.
We set: $I_1:=<\tilde K\star\phi,\psi_{\epsilon}\,\tilde w>$ 
and $I_2:=<\tilde K\star\phi,\psi'_{\epsilon}\,\tilde w>$
so that we have:
\begin{equation}
\label{artialtemperedchomolog56laplacientemperate41UVWX}
\vert < \tilde K\star\phi,\tilde w>\vert^2=\vert I_1+I_2\vert^2\le 2\vert I_1\vert^2+2\vert I_2\vert^2
\end{equation}

with  $I_1:=< \tilde K\star\phi,\psi_{\epsilon}\,\tilde w>=<\psi_{\epsilon}(y)\int_{x\in Supp \phi} \tilde K(y-x)\phi(x) d\lambda (x),\tilde w>_y$ 
and $I_2:=< \tilde K\star\phi,\psi'_{\epsilon}\,\tilde w>=<\psi'_{\epsilon}(y)\int_{x\in Supp \phi} \tilde K(y-x)\phi(x) d\lambda (x),\tilde w>_y$.\\
As $Supp\,\psi_{\epsilon}\subset\Omega$
 we have $\psi_{\epsilon}\,\tilde w=\psi_{\epsilon}\,w$
 ($w$ is in $L^{2,k}(\Omega)$ 
and $\tilde w$ is an extension of $w$ 
as a distribution in $\mathcal D'(\R^n)$),\\
and then  $I_1:=< \tilde K\star\phi,\psi_{\epsilon}\,\tilde w>=< \tilde K\star(\psi_{\epsilon}\, w),\phi>$.
We use Cauchy-Schwarz inequality and convolution inequality :\\
$\vert I_1\vert^2\le \vert\vert \tilde K\star(\psi_{\epsilon}\, w)\vert\vert_2^2 \,\,\vert\vert\phi\vert\vert_2^2
\le\vert\vert \tilde K\vert\vert_1^2 \,\,\vert\vert\psi_{\epsilon}\, w\vert\vert_{2}^2 \,\,\vert\vert\phi\vert\vert_{2}^2$.\\
$\vert I_1\vert^2
\le\vert\vert \tilde K\vert\vert_1^2 \,\,\vert\vert\psi_{\epsilon}\, w\vert\vert_{2}^2 \,\,\vert\vert\phi\vert\vert_{2}^2
\le\vert\vert \tilde K\vert\vert_1^2 \,(\frac{\epsilon}{4})^{-2k}\,\vert\vert\psi_{\epsilon}\, w\vert\vert_{2,k}^2 \,(2\epsilon)^{2(k+n-1)}\,\vert\vert\phi\vert\vert_{2,-k-n+1}^2$.\\
Indeed as $d_{\Omega}(x)\ge \frac{\epsilon}{4}$ on the support of $\psi_{\epsilon}$,
we have: $(\frac{\epsilon}{4})^{2k}\vert\vert\psi_{\epsilon}\, w\vert\vert_{2}^2\le\vert\vert\psi_{\epsilon}\, w\vert\vert_{2,k}^2$ 
and as $d_{\Omega}(x)\le 2\epsilon$ on $V_{\epsilon}$ 
which contains the support of $\phi$ ,
we have: $(2\epsilon)^{-2(k+n-1)}\vert\vert\phi\vert\vert_{2}^2\le\vert\vert\phi\vert\vert_{2,-k-n+1}^2$.\\
Finally we get: 
\begin{equation}
\label{artialtemperedchomolog56laplacientemperate52341cdgUVWX}
\vert I_1\vert^2 
\le 2^{6k+2n-2} \epsilon^{2n-2} \vert\vert \tilde K\vert\vert_1^2 \,\,\vert\vert w\vert\vert_{2,k}^2 \,\,\vert\vert\phi\vert\vert_{2,-k-n+1}^2
\end{equation}
 for every $\phi$ with $Supp\,\, \phi\subset V_{\epsilon}$.\\

 We now consider the term $I_2$. 
As $\tilde w$ is a distribution of order at most $k$ and has in support in the ball $L$,
 $\tilde w$ satisfies an inequality :
\begin{equation}
\label{artialtemperedchomolog56laplacientemperate52341UVWX}
\vert<\tilde w,\psi>\vert\le C(\tilde w)\,\sup_{y\in\, L} \Sigma_{\vert\alpha\vert\le k}
 \, \vert D_{y}^{\alpha}\psi(y)\vert,
\end{equation}

for all $\psi\in \mathrm{C}^{l}(\R^n)$
 with $C(\tilde w):=C(\Omega,L)\vert\vert w\vert\vert_k$ 
using  theorem \ref{w(z)=w(zeta),rhoepsilon9on} 
remark \ref{observethattheextensiontildew2}
(replacing $L$ by a compact neighborhood of $L$ if necessary).\\
We apply inequality (\ref{artialtemperedchomolog56laplacientemperate52341UVWX}) 
to the function $\psi (y):=\int_{x\in Supp\, \phi} \psi'_{\epsilon}(y) K(y-x)\phi(x) d\lambda (x)$ 
and we take the derivative in the variable $y$ under the symbol $\int$ 
so that we get :
\begin{equation}
\label{artialtemperedchomolog56laplacientemperate521UVWX}
\vert I_2)\vert\le C(\tilde w) \max_{y\in L\setminus K_{\frac{\epsilon}{4}}} 
\Sigma_{\vert\alpha\vert\le k}\,\vert\int_{x\in Supp\,\phi} D^{\alpha}_y\lbrack\psi'_{\epsilon}(y)\tilde K(y-x)\rbrack\,\phi(x) d\lambda (x)\;\vert
\end{equation}
Let us observe that $Supp\,\psi'_{\epsilon}\subset \R^n\setminus K_{\frac{\epsilon}{4}}$.\\
For $x\in\mathrm{Supp}\,\phi\subset V_{\epsilon}$ and $y\in\R^n\setminus K_{\frac{\epsilon}{4}}$ 
we have $\vert x-y\vert\ge \frac{3\epsilon}{4}>\frac{3}{8} d_{\Omega}(x)$
(there exists $z\in\partial\Omega$ such that $ d_{\Omega}(y)=\vert y-z\vert<\frac{\epsilon}{4}$, 
then $\vert x-z\vert\ge d_{\Omega}(x)>\epsilon$ 
so that
$\vert x-y\vert\ge \vert x-z\vert-\vert y-z\vert>\epsilon-\frac{\epsilon}{4}=\frac{3\epsilon}{4}$)
and then\\
$\vert D_y^{\beta}\tilde K(y-x)\vert=O(\vert x-y\vert^{-\vert\beta\vert-n+1})
=O(\lbrack d_{\Omega}(x)\rbrack^{-\vert\beta\vert-n+1})$. 
We also have $\vert D_y^{\gamma}\psi'_{\epsilon}(y)\vert=O(\epsilon^{-\vert\gamma\vert})
=O(\lbrack d_{\Omega}(x)\rbrack^{-\vert\gamma\vert})$ so that (as we have $\vert\beta\vert+\vert\gamma\vert=\vert\alpha\vert$) :
\begin{equation}
\label{artialtemperedchomolog56laplacientemperate6121UVWX} 
\vert D^{\alpha}_y\lbrack\psi'_{\epsilon}(y) \tilde K(y-x)\rbrack\vert
\le C_1(\Omega)\lbrack d_{\Omega}(x)\rbrack^{-\vert\alpha\vert-n+1}
\end{equation}
(\ref{artialtemperedchomolog56laplacientemperate521UVWX}) and (\ref{artialtemperedchomolog56laplacientemperate6121UVWX}) imply:
\begin{equation}
\label{artialtemperedchomolog56laplacientemperate523ret1UVX}
\vert I_2)\vert\le C(\Omega,\tilde w) 
\int_{x\in Supp\,\phi} \lbrack d_{\Omega}(x)\rbrack^{-k-n+1}\,\vert\phi(x)\vert d\lambda (x)
\end{equation}
Using the Cauchy-Schwarz inequality we have:
\begin{equation}
\label{artialtemperedchomolog56laplacientemperate523retschz1UVX}
\vert I_2\vert^2\le \lbrack C(\Omega,\tilde w)\rbrack^2(\lambda (Supp\,\phi))^2 
\int_{x\in Supp\,\phi} \lbrack d_{\Omega}(x)\rbrack^{-2k-2n+2}\,\vert\phi(x)\vert^2 d\lambda (x)
\end{equation}
and then:
\begin{equation}
\label{artialtemperedchomolog56laplacientemperate523retschz21UVWX}
\vert I_2\vert^2\le \lbrack C(\Omega, \tilde w)\rbrack^2(\lambda (V_{\epsilon}))^2 \vert\vert\phi\vert\vert_{2,-k-n+1}^2
\end{equation}
(as $Supp\,\, \phi\subset V_{\epsilon}$ and as $\vert\vert\phi\vert\vert_{2,-k-n+1}^2:=\int_{x\in V_{\epsilon}} \lbrack d_{\Omega}(x)\rbrack^{-2k-2n+2}\,\vert\phi(x)\vert^2 d\lambda (x))$.\\
Using inequalities (\ref{artialtemperedchomolog56laplacientemperate41UVWX}), (\ref{artialtemperedchomolog56laplacientemperate52341cdgUVWX}) (for $I_1$)
and (\ref{artialtemperedchomolog56laplacientemperate523retschz21UVWX}) (for $I_2$) we get:
\begin{equation}
\label{artialtemperedchomolog56laplacientemperate4sdghx1UVX}
\vert < K\star\tilde w,\phi>\vert^2\le C_1(\Omega,\tilde w)\lbrack(\lambda (V_{\epsilon}))^2 +\epsilon^{2n-2}\rbrack
\int_{x\in V_{\epsilon}} \lbrack d_{\Omega}(x)\rbrack^{-2k-2n+2}\,\vert\phi(x)\vert^2 d\lambda (x)
\end{equation}
for all $\phi$ with $Supp\,\phi\subset V_{\epsilon}$ and therefore:
\begin{equation}
\label{artialtemperedchomolog56laplacientemperate4sdghx21UVX}
\int_{V_{\epsilon}}\vert K\star\tilde w\vert^2\lbrack d_{\Omega}(x)\rbrack^{2k+2n-2}d\lambda\le C_1(\Omega,\tilde w)
 \lbrack(\lambda (V_{\epsilon}))^2 + \epsilon^{2n-2}\rbrack
\end{equation}
with $C_1(\Omega,\tilde w):=2\lbrack C(\Omega, \tilde w)\rbrack^2+22^{6k+2n-2} \vert\vert \tilde K\vert\vert_k^2 \,\,\vert\vert w\vert\vert_{2,k}^2
\le C_2(k,n,\Omega)\vert\vert w\vert\vert_{2,k}^2$.\\
As $\lbrack\lambda(V_{\epsilon})\rbrack^2\le \lambda(V_{\epsilon})\rbrack$, 
taking $\epsilon=2^{-j}$, $j\in\N$,
 summing on $j$ and setting $\Omega':=\bigcup_{j\in\N} V_{2^{-j}}$ we have:

\begin{equation}
\label{artialtemperedchomolog56laplacientemperate4sdghx2311UVWX}
\int_{\Omega'}\vert K\star\tilde w\vert^2\lbrack d_{\Omega}(x)\rbrack^{2k+2n-2}d\lambda\le C_1(\Omega,\tilde w)
\lbrack\lambda(\Omega')+2\rbrack, 
\end{equation}
Taking $\epsilon=\frac{3}{4}2^{-j}$, $j\in\N$,
 summing on $j$ and setting $\Omega'':=\bigcup_{j\in\N} V_{\frac{3}{4}2^{-j}}$ we have:

\begin{equation}
\label{artialtemperedchomolog56laplacientemperate4sdghx2312UVWX}
\int_{\Omega''}\vert K\star\tilde w\vert^2\lbrack d_{\Omega}(x)\rbrack^{2k+2n-2}d\lambda\le C_1(\Omega,\tilde w)
\lbrack\lambda(\Omega'')+2\rbrack, 
\end{equation}
Giving $j\in\N$, the set $K_{2^{-j}}\setminus \Omega_{2^{-j}}:=\,\lbrace x\in(\Omega)\vert\, d_{\Omega}(x)=2^{-j}\rbrace$
 is a subset of $\Omega''$ so that
 $\Omega\subset\Omega'\cup\Omega''$.  Summing (\ref{artialtemperedchomolog56laplacientemperate4sdghx2311UVWX}) 
and (\ref{artialtemperedchomolog56laplacientemperate4sdghx2312UVWX}) we finally have :

\begin{equation}
\label{artialtemperedchomolog56laplacientemperate4sdghx23113UVWX}
\int_{\Omega}\vert K\star\tilde w\vert^2\lbrack d_{\Omega}(x)\rbrack^{2k+2n-2}d\lambda\le 2 C_1(\Omega,\tilde w)
\lbrack\lambda(\Omega)+2\rbrack, 
\end{equation}

i.e. 
$\vert\vert v'\vert\vert_{k+n-1}^2\le
 2 C_1(\Omega,\tilde w)\lbrack\lambda(\Omega)+2\rbrack$. 
\end{proof}

We can now prove the following slight extension of theorem \ref{partialtemperedchomology}: the aim of which is to  iterate the construction made in $\C^n$.
\begin{proposition}
\label{partialtemperedchomologyescrtyX}
 Let $\Omega$ be a bounded Stein open subset of $\C^n$ 
and $f$ be a given current of bidegree $(p,q+1)$ on $\C^n$ with compact support 
which is $\bar\partial$-closed on $\Omega$ 
and can be written $f=g+h$ 
 with $g_{\vert\Omega}\in L^{2,k}_{(p,q+1)}(\Omega)$ 
and $h\in H^{-s}_{(p,q+1)}(\C^n)$ for some $s\le k$.
Then there exists a current $w=u+v$ of bidegree $(p,q)$ (with compact support) in $\C^n$ such that:
\begin{equation}\label{}
\bar\partial w=f,
\end{equation}
$in\,\Omega$,\\
with $w=u+v\in H_{(p,q)}^{-k-5n-2}(\C^n)$,
  $u_{\vert\Omega}\in L_{(p,q)}^{2,k+4n+1}(\Omega)$ 
	and $v\in H_{(p,q)}^{-s+1}(\C^n)$).
\end{proposition}

Let us observe that $v$ has a better regularity than $h$. $u$ has still a polynomial $L^2$ growth in $\Omega$  as $g$ (even if it is larger than that of $g$).

\begin{proof}
 Using  extension theorem \ref{w(z)=w(zeta),rhoepsilon9on} (remark \ref{observethattheextensiontildew1}),
 we can assume that $g$ and therefore $f$ are in $H^{-k-n-1}_{(p,q+1)}(\C^n)$.
At first we have to solve:
\begin{equation}\label{beanopenpseudoconvex10proposi}
\frac{1}{2}\Delta v=
\bar\partial^{\star} f=\bar\partial^{\star} g+\bar\partial^{\star} h.
\end{equation} 
We solve separately the equations $\frac{1}{2}\Delta v_1=\bar\partial^{\star} g$ in $\Omega$
and $\frac{1}{2}\Delta v_2==\bar\partial^{\star} h$ in $\C^n$ . 
Using lemma \ref{w(z)=w(zeta),rhoepsilon9on23avril18WX} (in $\R^{2n}$) 
we can at first find $v_1\in L_{(p,q)}^{2,k+2n-1}(\Omega)$ 
and then  using theorem \ref{w(z)=w(zeta),rhoepsilon9on} (remark \ref{observethattheextensiontildew1})
 we can find an extension  of $v_1$ to $\C^n$ 
which is a distribution in $H^{-k-3n}(\C^n)$.
We set $v:=v_1+v_2$. 
As $h\in H^{-s}(\C^n)$, $v_2$ is in $H^{-s+1}(\C^n)$ 
so that $v:=v_1+v_2\in H_{(p,q)}^{-k-3n}(\C^n)$.  
$f-\bar\partial v\in H_{(p,q+1)}^{-k-3n-1}(\C^n)$ is harmonic in $\Omega$  
and therefore (using lemma \ref{LetwbeasdistributiononRoforderk})  
$f-\bar\partial v$ is in $ L_{(p,q+1)}^{2,k+4n+1}(\Omega)$. 
Finally corollary \ref{beanopenpseudoconvex3} gives a solution 
$u\in L_{(p,q+1)}^{2,k+4n+1}(\Omega)$ 
of the equation $\bar\partial u=f-\bar\partial v$ in $\Omega$. 
Setting $w=u+v$, we get $\bar\partial w=f$ in $\Omega$. 
Using still theorem \ref{w(z)=w(zeta),rhoepsilon9on} (remark \ref{observethattheextensiontildew1})
we get an extension of $u$ in $H_{(p,q)}^{-k-5n-2}(\C^n)$
 so that $w\in H_{(p,q)}^{-k-5n-2}(\C^n)$.
\end{proof}
We need the following lemma making comparisons between the several distances to boundary we have to consider.
\begin{lemma}
\label{partialtemperedchomologyescrtynlminX}
Let $\Omega$ and $\Omega_j$ 
be two bounded open subsets of the complex Riemannian  manifold $X$ 
such that $\Omega\cap\Omega_j\neq\emptyset$ 
and $u\in L^{2,k}_{(p,q)}(\Omega)$. Then:
 \begin{equation}\label{avedOmegacapegageminpsijwedgeX}
\int_{\Omega\cap\Omega_j}\vert u\vert^2 \lbrack d_{\Omega\cap\Omega_j}(z)\rbrack^{2k}d\lambda\le
\int_{\Omega}\vert  u\vert^2 \lbrack d_{\Omega}(z)\rbrack^{2k}d\lambda.
\end{equation}
If $\psi_j\in\mathcal D(\Omega_j)$, and $u\in L^{2,k}_{(p,q)}(\Omega\cap\Omega_j)$ then:
\begin{equation}\label{avedOmegacapegageminpsijghtkX}
\int_{\Omega}\vert \psi_j u\vert^2 \lbrack d_{\Omega}(z)\rbrack^{2k}d\lambda\le\epsilon_j^{-2k}
\int_{\Omega\cap\Omega_j}\vert \psi_j u\vert^2 \lbrack d_{\Omega\cap\Omega_j}(z)\rbrack^{2k}d\lambda,
\end{equation}
and
\begin{equation}\label{avedOmegacapegageminpsijghtk1X}
\int_{\Omega}\vert \bar\partial\psi_j\wedge u\vert^2 \lbrack d_{\Omega}(z)\rbrack^{2k}d\lambda\le\epsilon_j^{-2k}
\int_{\Omega\cap\Omega_j}\vert \bar\partial\psi_j\wedge u\vert^2 \lbrack d_{\Omega\cap\Omega_j}(z)\rbrack^{2k}d\lambda,
\end{equation}
 with $\epsilon_j:=\min_{z\in Supp\,\psi_j} d_{\Omega_j}(z)$.

\end{lemma}
Let us observe that $Supp\,(\psi_j u)\subset \Omega\cap \Omega_j$ and $Supp\,(\bar\partial\psi_j\wedge u)\subset \Omega\cap \Omega_j$  
so that integrating $\psi_j u$ (resp. $\bar\partial\psi_j\wedge u$) on $\Omega$ or on $\Omega\cap \Omega_j$ 
gives the same result 
but we will need to consider later this extension of $\psi_j u$ to $\Omega$ 
(by zero outside its support).

\begin{proof}  For $z\in\Omega\cap\Omega_j$,
 we have: 
$d_{\Omega\cap\Omega_j}(z)=\min(d_{\Omega}(z), d_{\Omega_j}(z))\le d_{\Omega}(z)$ 
and therefore:
\begin{equation}\label{avedOmegacapegageminpsij1ome}
\int_{\Omega\cap \Omega_i}\vert u\vert^2\lbrack d_{\Omega\cap\Omega_i}(z)\rbrack^{2k}d\lambda
\le\int_{\Omega}\vert u\vert^2\lbrack d_{\Omega}(z)\rbrack^{2k}d\lambda 
\end{equation}
 
On the other hand, we define $\epsilon_j:=\min_{z\in Supp\,\psi_j} d_{\Omega_j}(z)=\min_{z\in Supp\,\psi_j} d(z,\partial\Omega_j)$
 ($0<\epsilon_j\le 1$)
 so that for $z\in Supp\, \psi_j$, 
we have: $d_{\Omega\cap\Omega_j}(z)=\min (d_{\Omega}(z),d_{\Omega_j}(z))\ge\min (d_{\Omega}(z),\epsilon_j)\ge\epsilon_j d_{\Omega}(z)$ (as $d_{\Omega}(z)\le 1$)
 and then:
\begin{equation}\label{avedOmegacapegageminpsijerlvrty}
\int_{\Omega\cap \Omega_j}\vert \psi_j u\vert^2 \lbrack d_{\Omega}(z)\rbrack^{2k}d\lambda
\le
\epsilon_j^{-2k}
\int_{\Omega\cap \Omega_j}\vert \psi_j u\vert^2 \lbrack d_{\Omega\cap\Omega_j}(z)\rbrack^{2k}d\lambda
\end{equation}
and :
\begin{equation}\label{avedOmegacapegageminpsijerlvrtyvariante}
\int_{\Omega\cap \Omega_j}\vert \bar\partial\psi_j\wedge u\vert^2 \lbrack d_{\Omega}(z)\rbrack^{2k}d\lambda
\le
\epsilon_j^{-2k}
\int_{\Omega\cap \Omega_j}\vert\bar\partial \psi_j\wedge u\vert^2 \lbrack d_{\Omega\cap\Omega_j}(z)\rbrack^{2k}d\lambda.
\end{equation}

\end{proof}

We can now prove the following result by the same reasoning as in the case of $\C^n$. 
Moreover we consider currents in $\mathcal D'_{(p,q)}(X,F)$  with values 
in a given holomorphic vector bundle $F$
(to simplify we only consider current in $H^{-k}(X,F)$, $k\in\N$).

\begin{theorem}
\label{partialtemperedchomologyStein2}
 Let $\Omega$ be a relatively compact open Stein subset of a Stein manifold $X$ and $F$ be a given Hermitian holomorphic vector bundle on $X$.
Then for every current $f$ of bidegree $(p,q+1)$ on $X$ with values in $F$ (and with compact support in $X$) 
which is $\bar\partial$-closed on $\Omega$,
 there exists a current $w$ of bidegree $(p,q)$ on $X$ with values in $F$ (with compact support) such that:
\begin{equation}
\label{artialtemperedchomolog56}
\bar\partial w=f,
\end{equation}
in $\Omega$. Moreover if $f$ is in $H^{-k}_{(p,q+1)}(X,F)$ for some  $k\in\N$, we can find a solution $w$
in $H_{(p,q)}^{-n-1-r}(X,F)$ with $r=k(4n+2)-3n-1$, 
more precisely $w=u+v$, $u_{\vert\Omega}\in L_{(p,q)}^{2,r}(\Omega,F)$, 
$u\in H_{(p,q)}^{-n-1-r}(X,F)$
 and $v\in H_{(p,q)}^{-k+1}(X,F)$.
\end{theorem}

\begin{proof}
 
Let us assume that $f\in H^{-k}_{(p,q+1)}(X, F,loc)$. 
We will prove that we can find a solution $u\in H^{-n-1-r}_{(p,q)}(X, F,loc)$.\\
By considering local charts of $X$, 
we will locally reduce the problem to the case of $\C^n$ 
and using a partition of the unity, 
we will patch together these local solutions 
to obtain a first approximate global solution $\tilde w_1$ 
such that $f_1:=f-\bar\partial \tilde w_1$ has in some sense better regularity than $f$ 
and then we iterate the construction (replacing $f$ by $f_1$)
until $f_1$ becomes enough regular so that we can use J-P. Demailly's $L^2$-estimate for $\bar\partial$
 (i.e. until we have found in the temperate cohomology class of $f$ a new current $f_1$ such that ${f_1}_{\vert\Omega}\in L^{2,k+r}_{(p,q+1)}(\Omega, F)$ for some $r$ large enough).\\
 Using local charts on $X$ and Borel-Lebesgue lemma,
 we can find a finite open covering of the compact set $\bar\Omega$
 by relatively compact open subsets $\Omega_j$ of $X$, $1\le j\le N$ 
such that every $\bar \Omega_j$ is  contained in a geodesic chart for the given Riemannian metric 
and every $\Omega_j$ is biholomorphic to a bounded open ball $U_j:=B_j(z_j, r_j)$ of $\C^n$,
by a local biholomorphic map $\phi_j$ 
defined on a neighborhood of $\bar\Omega_j\subset X$ 
and taking its values into $\C^n$ 
($z_j\in\phi_j(\bar\Omega)$, $r_j>0$). 
Moreover we can also suppose (by shrinking enough each $\Omega_j$)
 that the exponential map sending
 the tangent space $T_{z_j} X$ (of $X$ at $z_j$) into $X$ 
is a diffeomorphism of a open ball in $T_{z_j} X$ 
onto a geodesic open ball of center $z_j$ containing $\bar\Omega_j$ so that the geodesic  distance 
and the Euclidian distance coming from $\C^n$ (by means of $\phi_j$) are equivalent on a neighborhood of $\bar\Omega_j$
 and so that the spaces $L^{2,k}_{(p,q)}(\Omega_j\cap\Omega, F)$  ($k\in \N$) associated with the geodesic distance to $\partial(\Omega_j\cap\Omega)$ 
or with the Euclidian distance to $\partial(\Omega_j\cap\Omega)$ coming from $\C^n$ (by means of $\phi_j$) are the same.
Finally we can also suppose that the given holomorphic vector bundle $F$ is trivial 
on a neighborhood of each $\bar\Omega_j$.\\
For every bounded Stein open subset $\phi_j(\Omega\cap\Omega_j)\subset \phi_j(\Omega_j)=:U_j\subset\C^n$ 
we use the construction made in the case of $\C^n$ 
(\emph{i.e.} remark \ref{Proofoftheorem2remarkprecise}, 
$F$ is trivial on a neighborhood of $\bar\Omega_j$)
so that we can construct $u_j\in H^{-k-2n-1}_{(p,q)}(\Omega_j,F)$ 
with ${u_j}_{\vert\Omega\cap\Omega_j}\in L^{2,k+n}_{(p,q)}(\Omega_j\cap\Omega, F)$ 
and (it is the key point)  $v_j\in H^{-k+1}_{(p,q)}(\Omega_j, F)$ 
 such that $w_j:=u_j+v_j$ is a solution 
of $f=\bar\partial w_j=\bar\partial (u_j+ v_j)$ in $\Omega\cap\Omega_j$ 
and $w_j\in H^{-k-2n-1}_{(p,q)}(\Omega_j, F)$.\\
 Let $\psi_j\in\mathcal D(\Omega_j))\ge 0$ be a partition of unity such that:
$\sum_{j=1}^{j=N} \psi_j =1$ on a neighborhood of $\bar\Omega$.
 Then $\psi_j w_j\in H^{-k-2n-1}_{(p,q)}(X, F)$
 and  a key point is that ${\psi_j u_j}_{\vert\Omega}\in L^{2,k+n}_{(p,q)}(\Omega, F)$ 
for the distance $d_{\Omega}(z)$ 
using lemma \ref{partialtemperedchomologyescrtynlminX} inequality  (\ref{avedOmegacapegageminpsijghtkX})
 (we apply with $u=u_j$)
and then  for $i\neq j$, ${\psi_j u_j}_{\vert\Omega\cap\Omega_i}\in L^{2,k+n}_{(p,q)}(\Omega\cap\Omega_i, F)$ 
for the distance $d_{\Omega\cap\Omega_i}(z)$ 
using lemma \ref{partialtemperedchomologyescrtynlminX} inequality (\ref{avedOmegacapegageminpsijwedgeX}) 
(we apply with $u=\psi_j u_j$ and to $\Omega_i$ instead of $\Omega_j$).
 Particularly $({\sum_{j=1}^{j=N}\psi_j u_j})_{\vert\Omega}\in L^{2,k+n}_{(p,q)}(\Omega, F)$ 
for the distance $d_{\Omega}(z)$ 
and for all $1\le i\le N$, $({\sum_{j=1}^{j=N}\psi_j u_j})_{\vert\Omega\cap\Omega_i}\in L^{2,k+n}_{(p,q)}(\Omega\cap\Omega_i, F)$ 
for the distance $d_{\Omega\cap\Omega_i}(z)$.

 Gluing together the local solutions $w_j$, 
we define $\tilde w_1 =\sum_{j=1}^{j=N}\psi_j w_j=\tilde u_1+\tilde v_1$, 
with $\tilde u_1 =\sum_{j=1}^{j=N}\psi_j u_j$ 
and  
$\tilde v_1 =\sum_{j=1}^{j=N}\psi_j v_j$
 so that
$\tilde w_1\in H^{-k-2n-1}_{(p,q)}(X, F)$,
$\tilde {u_1}_{\vert\Omega}\in L^{2,k+n}_{(p,q)}(\Omega, F)$, 
$\tilde v_1\in H^{-k+1}_{(p,q)}(X, F)$ 
and we obtain:
\begin{equation}\label{artialtemperedchomolog5674}
\bar\partial \tilde w_1= \sum_{j=1}^{j=N}\psi_j\bar\partial w_j+\sum_{j=1}^{j=N}(\bar\partial \psi_j)\wedge w_j.
\end{equation}
 We have: $\bar\partial w_j=f$ in $\Omega\cap\Omega_j$  and  $Supp\,\psi_j\subset \Omega_j$ 
so that equation (\ref{artialtemperedchomolog5674}) implies in $\Omega$ :
\begin{equation}\label{artialtemperedchomolog5684}
\bar\partial\tilde w_1=(\sum_{j=1}^{j=N}\psi_j)f+\sum_{j=1}^{j=N}\bar\partial\psi_j\wedge w_j,
\end{equation}
 or (in $\Omega$) :
\begin{equation}\label{artialtemperedchomolog5684parz}
\bar\partial\tilde w_1=f+\sum_{j=1}^{j=N}\bar\partial\psi_j\wedge w_j.
\end{equation} 
We set:
$f_1:=-\sum_{j=1}^{j=N}\bar\partial\psi_j\wedge w_j$, 
$g_1:=-\sum_{j=1}^{j=N}\bar\partial\psi_j\wedge u_j$ 
and $h_1:=-\sum_{j=1}^{j=N}\bar\partial\psi_j\wedge v_j$,
 so that we obtain:
\begin{equation}\label{artialtemperedchomolog5694}
f_1=f-\bar\partial \tilde w_1=g_1+h_1.
\end{equation} 

$f_1$ is $\bar\partial$-closed  on $\Omega$ 
and we have $f_1:=g_1+h_1$ with $g_1\in H^{-k-2n-1}_{(p,q+1)}(X, F)$ and $h_1\in H^{-k+1}_{(p,q+1)}(X, F)$ (as each $v_j\in H^{-k+1}_{(p,q)}(\Omega_j, F)$).
Moreover according to lemma \ref{partialtemperedchomologyescrtynlminX}, 
${g_1}_{\vert\Omega}\in L^{2,k+n}_{(p,q+1)}(\Omega, F)$ 
for the distance $d_{\Omega}(z)$ (inequality (\ref{avedOmegacapegageminpsijghtk1X})) 
and for all $i$,
  ${g_1}_{\vert\Omega\cap\Omega_i}\in L^{2,k+n}_{(p,q+1)}(\Omega\cap\Omega_i, F)$,
	for the distance $d_{\Omega\cap\Omega_i}(z)$ 
	(inequality (\ref{avedOmegacapegageminpsijwedgeX}) 
	applied in $\Omega_i$ to each form $\bar\partial\psi_j \wedge u_j$ of bidegre (p,q+1) with $j\neq i$).
Let us observe that $g_1$ or $h_1$ 
are not necessarly $\bar\partial$-closed (that is the difficulty).
The key point is that $f_1$ is better than $f$ 
in the sense that $h_1\in H^{-k+1}_{(p,q+1)}(X, F)$ has a best regularity than $f\in H^{-k}_{(p,q+1)}(X, F)$  
and $g_1$ has a $L^2$ polynomial growth in $\Omega$.\\
We have built a first approximate global solution $\tilde w_1$ 
such that $f_1:=f-\bar\partial \tilde w_1$ is better than $f$. 
Applying proposition \ref {partialtemperedchomologyescrtyX} and lemma
\ref{partialtemperedchomologyescrtynlminX}
to each Stein open subset $\phi_j(\Omega\cap\Omega_j)\subset U_j\subset\C^n$, 
we can iterate this construction 
and construct a finite sequence of currents on $X$: $f_{l+1}=g_{l+1}+h_{l+1}=f_l-\bar\partial \tilde w_{l+1}$, $l\in\N$, $f_0:=f$,
such that $h_{l+1}\in H^{-k+l+1}_{(p,q+1)}(X, F)$ has a strictly better regularity 
than $h_l\in H^{-k+l}_{(p,q+1)}(X, F)$ 
and ${g_{l+1}}_{\vert\Omega}\in L^{2,k+n+l(4n+1)}_{(p,q+1)}(\Omega, F)$ 
 has  still a $L^2$ polynomial growth in $\Omega$, moreover $g_{l+1}\in H^{-k-2n-1-l(4n+1)}_{(p,q+1)}(X, F)$. 
Given $f_l=g_l+h_l$, we apply proposition \ref {partialtemperedchomologyescrtyX} to each $\phi_j(\Omega\cap\Omega_j)\subset\C^n$, 
we construct a solution $w_{l+1,j}=u_{l+1,j}+v_{l+1,j}$ of $\bar\partial w_{l+1,j}=f_l$ in $\Omega\cap\Omega_j\subset X$
 ($\Omega\cap\Omega_j$ is biholomorphic to $\phi_j(\Omega\cap\Omega_j)$).\\
We set  $\tilde w_{l+1}:=\sum_{j=1}^{j=N}\psi_j w_{l+1,j}=\tilde u_{l+1} + \tilde v_{l+1}$ 
  with $\tilde u_{l+1}:=\sum_{j=1}^{j=N}\psi_j u_{l+1,j}$,
 $\tilde v_{l+1}:=\sum_{j=1}^{j=N}\psi_j v_{l+1,j}$  
and then $f_{l+1}:=f_l-\bar\partial \tilde w_{l+1}=-\sum_{j=1}^{j=N}\bar\partial\psi_j\wedge w_{l+1,j}=g_{l+1}+h_{l+1}$, 
 with $g_{l+1}:=-\sum_{j=1}^{j=N}\bar\partial\psi_j\wedge u_{l+1,j}$ 
and $h_{l+1}:=-\sum_{j=1}^{j=N}\bar\partial\psi_j\wedge v_{l+1,j}$.\\
The estimates we need to iterate are straightforward consequences 
of proposition
 \ref {partialtemperedchomologyescrtyX} 
(we use replacing $k$ by $k+n+(l-1)(4n+1)$ and $s$ by $k-l$ at the step $l\ge1$)
and lemma
\ref{partialtemperedchomologyescrtynlminX}.
Of course the polynomial $L^2$ growth of $g_{l+1}$ is larger than that of $g_l$ 
but it does'nt matter 
as we are able to solve the $\bar\partial$ equation in $\Omega$ 
for any $L^2$ polynomial growth.\\
Finally for $l=k$ (it is the last key point)  $h_{k}\in H_{(p,q+1)}^{0}(X, F)=L_{(p,q+1)}^{2,0}(X, F)$,
particularly  ${h_{k}}_{\vert\Omega}\in L_{(p,q+1)}^{2,0}(\Omega, F)$,
so that $f_{k}$ verifies the $L^2$-estimate we need: 
$f_{k}=g_{k}+h_{k}$ has polynomial $L^2$ growth in $\Omega$ : 
$\int_{\Omega}\vert f_{k}\vert^2  \lbrack d_{\Omega}(z)\rbrack^{2r}  d\lambda<+\infty$
 with $r=k+n+(k-1)(4n+1)=k(4n+2)-3n-1$.
 We can now use J-P. Demailly's theorem \ref{beanopenpseudoconvex3SteinDem12}
to solve  in $\Omega$ the equation $\bar\partial \tilde w_{k+1}=f_{k}$
 with  $\tilde w_{k+1}\in L^{2,r}_{(p,q)}(\Omega,F)$
so that we obtain: $f= \bar\partial \tilde w$ 
with $\tilde w:= \tilde w_{k+1}+\Sigma_{l=1}^{l=k} \,\tilde w_{l}=u+v$, 
 $u:= \tilde w_{k+1}+\Sigma_{l=1}^{l=k} \,\tilde u_{l}$ 
and $v:= \Sigma_{l=1}^{l=k} \,\tilde v_{l}$, $u\in L^{2,r}_{(p,q)}(\Omega, F)$, $v\in H^{-s+1}_{(p,q)}(X,F)$. 
Using theorem \ref{w(z)=w(zeta),rhoepsilon9on} 
we get an extension of $u$ in $H_{(p,q)}^{-n-1-r}(X,F)$
 so that $\tilde w\in H_{(p,q)}^{-n-1-r}(X,F)$. 
\end{proof}

As explained in P. Schapira's article [Scha2020] (remark 2.3.4), 
theorem \ref{partialtemperedchomologyStein2} implies the following result 
 (theorem 2.3.3 in [Scha2020]). 
We refer to [Scha2020] for the definitions of the (derived) sheave
$\mathcal O_{X_{\mathrm{sa}}}^{\mathrm {tp}}$ 
of temperate holomorphic functions 
(defined on the subanalytic site $X_{\mathrm{sa}}$)
 and of other objects associated with.

\begin{theorem}
\label{partialtemperedchomologyStein2PS}\rm{(P. Schapira)}
 Let $X$ be a complex Stein manifold and let $\Omega$ be a subanalytic relatively compact Stein open subset of $X$ 
contained in a Stein compact subset $K$ of $X$. Let $\mathcal F$ be a coherent $\mathcal O_X$-module defined on a neighborhood of $K$.
Then  $\mathrm{R}\Gamma(\Omega;\mathcal F^{\mathrm{tp}})$ is concentrated in degree 0.

\end{theorem}

	\bibliographystyle{apalike}
\bibliography{biblio} 
[Dem1982] Demailly, J.P., 
\textit{Estimations $\mathrm {L}^2$ pour l'op\'erateur $\bar{\partial }$ d'un fibr\'e vectoriel holomorphe semi-positif au-dessus d'une vari\'et\'e k\"ahl\'erienne compl\`ete}
in \textit{Annales scientifiques de l' \'Ecole Normale Sup\'erieure} \textbf{15}, 457-511, (1982).

[Dem2012] Demailly, J.P.,\textit{ Complex Analytic and Differential Geometry.} Open Content Book. Chapter VIII, paragraph 6, Theorem 6.9, p. 379 (2012).

[Dol1956]	Dolbeault, P., \textit{Formes diff\'erentielles et cohomologie sur une vari\'et\'e analytique complexe}, I, in \textit{Ann. of Math.} 64, 83-130, (1956); II, in \textit{Ann. of Math.} \textbf{65}, 282-330, (1957). 

[Ele1975] Elencwajg, G.,\textit{Pseudoconvexit\'e locale dans les vari\'et\'es k\"ahl\'eriennes,} in \textit{Annales de l'Institut Fourier,} \textbf{25}, 295-314, (1975).

[H\"or1965] H\"ormander, L., $L^2$\textit{Estimates and Existence Theorem for the $\bar\partial$ \textit{Operator.}}, in \textit{Acta Math.}, \textbf{113}, 89-152 (1965). 

[H\"or1966] H\" ormander, L., \textit{An Introduction to Complex Analysis in Several Variables,} (1966), 3d edition,
 \textit{North Holland Math. Libr.}, Vol.\textbf{7}, Amsterdam, London, (1990).

[H\"or1983] H\" ormander, L., \textit{The Analysis of Linear Partial Diffenrential Operators I},
 \textit{Grundlehren der Mathematischen Wissenschhaften}, Vol. \textbf{256}, Springer Verlag, Berlin, Heidelberg, New York, Tokyo (1983).

[KS1996] Kashiwara, M. and Schapira, P., \textit{Moderate and Formal Cohomology associated with Constructibles Sheaves}, in \textit{M\'emoires Soc. Math. France,} Vol \textbf{64}, 76 pp, (1996).

[Le1964] Lelong, P., \textit{Fonctions enti\`eres (n variables) et fonctions plurisousharmoniques d'ordre fini dans $\C^n$}, \textit{Journal d'Analyse (Jerusalem)}, \textbf{12}, 365-406 (1964).  

[Scha2017] Schapira, P., \textit{Microlocal Analysis and beyond,} in \textit{arXiv}. 1701.08955 [math] (Pdf) (2017)

[Scha2020] Schapira, P., \textit{Vanishing of Temperate Cohomology on Complex Manifolds}, in \textit{arXiv.} 2003.11361 [math CV] (Pdf) (2020) 

[Schw1950] Schwartz, L., \textit{Th\'eorie des distributions.} Tome I, in \textit{Actualit\'es Sci. Ind.}, vol. 1091, Publ. Inst. Math. Univ. Strasbourg, no IX, Paris: Hermann \& Cie., (1950), 2e \'ed. (1957); r\'e\'ed. (1966).

[Siu1970] Siu, Y.T., \textit{Holomorphic Functions of Polynomial Growth on Bounded Domains.}, in \textit{Duke Mathematical Journal}, \textbf{37}, 77-84, (1970).

[Sk1971] Skoda, H., \textit{$d^{\prime \prime } $-cohomologie \`a croissance lente dans C$^{n}$}, in  \textit{Annales scientifiques de l'\'Ecole Normale Sup\'erieure,} S\'erie 4, Tome 4, no. 1, 97-120 (1971).

[Sk2020] Skoda, H., \textit{A Dolbeault lemma for temperate currents}, in \textit{arXiv.} 2003.11437 v2 [math CV] (Pdf) (2020)

[Weil1958] Weil, A., \textit{Introduction \`a l'\'etude des vari\'et\'es k\"ahl\'eriennes}, \textit{Actualit\'es scientifiques et industrielles}, 1267, Hermann, Paris, (1958).

\end{document}